\documentclass[oneside, hidelinks, 12pt]{article}
\usepackage{hyperref}
\usepackage{amsmath, amsthm, amssymb}
\usepackage{graphicx}
\usepackage{caption}
\usepackage{mathbbol}
\usepackage{txfonts}
\usepackage[mathscr]{eucal}
\usepackage{amssymb,latexsym}
\usepackage{verbatim}
\usepackage{graphicx}
\usepackage{epstopdf}
\usepackage{amsmath}
\usepackage{amsthm}
\usepackage{enumerate}
\usepackage{framed}
\usepackage{authblk}
\usepackage{setspace}
\usepackage[usenames]{color}
\definecolor{dgray}{RGB}{90,90,90}
\definecolor{gray}{RGB}{120,120,120}
\definecolor{lgray}{RGB}{150,150,150}

\theoremstyle{plain}
\newtheorem{thm}{Theorem}[section]
\newtheorem{lem}[thm]{Lemma}
\newtheorem{prop}[thm]{Proposition}
\newtheorem{cor}[thm]{Corollary}

\theoremstyle{definition}
\newtheorem{Def}[thm]{Definition}

\newtheorem{exm}[thm]{Example}

\theoremstyle{remark}
\newtheorem{rem}[thm]{Remark}

\newcommand{\field}[1]{\mathbb{#1}}

\newcommand\dhat{\hat{d}_\tau}

\renewcommand\S{\Sigma}
\newcommand\s{\sigma}
\renewcommand\d{\partial}

\newcommand\e{\epsilon}
\renewcommand\b{\beta}

\newcommand\g{\gamma}
\newcommand\8{\infty}
\renewcommand\a{\alpha}

\newcommand\beq{\begin{equation}}
\newcommand\eeq{\end{equation}}
\newcommand\be{\begin{equation}}
\newcommand\ee{\end{equation}}
\newcommand\ben{\begin{enumerate}}
\newcommand\een{\end{enumerate}}
\newcommand\bit{\begin{itemize}}
\newcommand\eit{\end{itemize}}

\makeatletter

\newcommand{\Rmnum}[1]{\expandafter\@slowromancap\romannumeral #1@}

\makeatother

\newcounter{mnotecount}[section]

\begin{document}

\title{Null distance on a spacetime}
\date{}

\author{Christina Sormani\thanks{C. Sormani began this research while in residence at the Mathematical Sciences Research Institute (MSRI) funded by NSF Grant No. 0932078000. Her research is funded in part by a PSC CUNY grant and NSF DMS 1309360.}  \, and Carlos Vega\thanks{C. Vega began this research as a postdoc at MSRI funded by NSF Grant No. 0932078 000, and continued this research as a postdoc at CUNY funded by Professor Sormani's NSF grants DMS - 1006059 and DMS - 1309360.} }

\maketitle

\begin{abstract}
Given a time function $\tau$ on a spacetime $M$, we define a \emph{null distance function}, $\hat{d}_\tau$, built from and closely related to the causal structure of $M$. In basic models with timelike $\nabla \tau$, we show that 1) $\hat{d}_\tau$ is a definite distance function, which induces the manifold topology, 2) the causal structure of $M$ is completely encoded in $\hat{d}_\tau$ and $\tau$. In general, $\hat{d}_\tau$ is a conformally invariant pseudometric, which may be indefinite. We give an `anti-Lipschitz' condition on $\tau$, which ensures that $\hat{d}_\tau$ is definite, and show this condition to be satisfied whenever $\tau$ has gradient vectors $\nabla \tau$ almost everywhere, with $\nabla \tau$ locally `bounded away from the light cones'. As a consequence, we show that the cosmological time function of \cite{AGHcosmo} is anti-Lipschitz when `regular', and hence induces a definite null distance function. This provides what may be interpreted as a canonical metric space structure on spacetimes which emanate from a common initial singularity, e.g. a `big bang'.
\end{abstract}

\newpage
\renewcommand\contentsname{}
\setcounter{tocdepth}{2}
\tableofcontents

\section{Introduction}

A basic distinction between Lorentzian and Riemannian geometry is the fact that Lorentzian manifolds are not known to carry an intrinsic distance function. (The standard Lorentzian `distance' function, reviewed below, is such in name only.) While any manifold is metrizable, what is desired more specifically is a distance function which captures a sufficient amount of the geometric structure. Among other applications, one motivation for finding such a distance function is the question of convergence in the Lorentzian setting, and of taking meaningful limits of sequences of spacetimes.

In this paper, we introduce a distance function on spacetime, which is positive-definite under natural conditions, and yet closely related to the causal structure. Indeed, in basic model cases, it encodes the causal structure completely. We begin first with a brief description of this distance, and summary of its main features.

Let $M$ be a spacetime and $\tau$ a time function on $M$. (Hence, $\tau$ increases along future causal curves.) Fix any two points $p, q \in M$. Let $\b$ be a `piecewise causal' curve from $p$ to $q$, that is, $\b$ is composed of subsegments which are either future or past causal, but $\b$ itself may wiggle backwards and forwards in time. (See Figure \ref{nulldist}.)
\begin{figure} [h]
\begin{center}
\includegraphics[width=13cm]{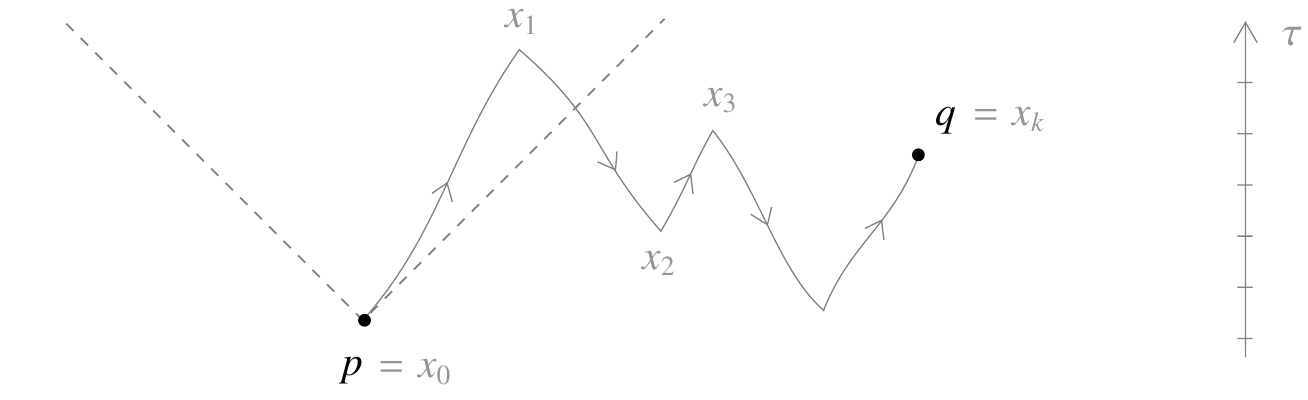}
\caption[...]{Measuring the `null distance' between $p$ and $q$. \label{nulldist}} 
\end{center}
\end{figure}
The `\emph{null length}' of $\b$ is computed by adding up the changes in $\tau$, in absolute value, along each of the causal subsegments of $\b$. Hence, letting $x_0, x_1, ..., x_k$ denote its breaks, the null length of $\b$ is given by $\hat{L}_\tau(\b) = \sum_{i=1}^k|\tau(x_i) - \tau(x_{i-1})|$. The infimum of this over all such curves is the \emph{null distance}, $\hat{d}_\tau(p,q)$. (The definitions are presented formally in Section \ref{secDef}.)

We note first that if $q$ is in the future of $p$, then any causal curve $\a$ from $p$ to $q$ is minimal, with $\hat{d}_\tau(p,q) = \hat{L}_\tau(\a) = \tau(q) - \tau(p)$. In particular, causal curves are distance-realizing, and in general we have:
\be
p \le q \; \; \Longrightarrow \; \; \hat{d}_\tau(p,q) = \tau(q) - \tau(p) \label{causality}
\ee
On the other hand, if $p$ and $q$ are not causally related, any timelike subsegments in $\b$ will necessarily be inefficient. If a minimal $\b$ does exist in this case, it must be piecewise \emph{null}. The name `null distance' derives primarily from this fact.

For $\tau = t$ on Minkowski space, $\hat{d}_\tau$ is a definite distance function, whose spheres are coordinate cylinders with axes in the time direction. Moreover, the converse of \eqref{causality} holds, and hence the causal structure is completely encoded via: 
\be
p \le q \; \; \Longleftrightarrow \; \; \hat{d}_\tau(p,q) = \tau(q) - \tau(p) \label{causalityencoded}
\ee
(See Proposition \ref{Mink}. This situation is also extended to more general warped product models in Theorem \ref{warpedthm}.) 

On the other hand, taking $\tau = t^3$ on Minkowski induces a null distance function which is indefinite. In this case, $\hat{d}_\tau$ assigns a distance of 0 between any two points in the $\{t=0\}$ slice, (where $\nabla \tau$ vanishes). Moreover, as a consequence of this, $\hat{d}_\tau$ can not distinguish causal relation for points on opposite sides of this slice, (see Proposition \ref{tcubed}). (Indeed, for \eqref{causalityencoded} to hold, it is necessary that $\hat{d}_\tau$ be definite.)

Evidently, the resolution of $\hat{d}_\tau$ with respect to separation of points, and causality, is closely related to the question of whether the gradient $\nabla \tau$ remains timelike, for smooth $\tau$. More generally, this question may be rephrased in terms of a lower bound on the `causal average rate of change' of $\tau$, which is realized below in the form of an `anti-Lipschitz' condition, (cf. Definition \ref{revLip}). A closely related condition is used by Chru{\'s}ciel, Grant, and Minguzzi in \cite{difftime}. Indeed, these two conditions are shown to be equivalent below. (As noted in \cite{difftime}, similar conditions were previously considered by Seifert in \cite{Seifertcosmic}.)

A time function $\tau$ which is locally anti-Lipschitz induces a definite null distance function, and we show this condition to be satisfied whenever $\tau$ has gradient vectors $\nabla \tau$ almost everywhere, with $\nabla \tau$ locally `bounded away from the light cones', (here quantified via Definition \ref{bddaway}). As a consequence, we show that the cosmological time function, defined by Andersson, Galloway, and Howard in \cite{AGHcosmo}, is locally anti-Lipschitz when `regular', and hence induces a definite null distance. For spacetimes $(M,g)$ emanating from a common initial singularity, cosmological time may be viewed as canonical, (it is determined by $g$), and hence its induced null distance function provides a uniform way of metrizing such spacetimes.

In general, null distance is closely related to the causal structure, and in basic model cases encodes causality completely via \eqref{causalityencoded}. We would expect that some (local) version of \eqref{causalityencoded} should hold more generally, under natural conditions on $\tau$, if not necessarily when $\hat{d}_\tau$ is definite. (See also Theorem \ref{warpedthm2}.) This is currently an open question, and will be explored in future work.

Strictly speaking, in the definition of $\hat{d}_\tau$, we require only that $\tau$ be strictly increasing along future causal curves, but not necessarily continuous, i.e., that $\tau$ be a `generalized time function'. Hence, the class of spacetimes considered here does include those which are stably causal, but also those which are (only) strongly causal, or for example, (only) past-distinguishing. Of course, in the stably causal, or further globally hyperbolic settings, $M$ admits smooth time functions $\tau$, with further geometric properties. Indeed, a study of $\hat{d}_\tau$ in such cases would be among several natural next steps. We will presently, however, remain more tightly focused on basic properties and examples, including a broad and detailed study of definiteness, as described above. Naturally, this issue is fundamental to a basic understanding of the null `distance', $\hat{d}_\tau$. But furthermore, this is also carried out for practical reasons, as part of the effort to understand the case of cosmological time. We emphasize that, within its scope, cosmological time is of special importance, not only because it is uniquely and `canonically' determined, but because it is so immediately tied to the geometry of the spacetime. 

We note finally that some care has been taken to keep the treatment below as accessible and self-contained as possible. In particular, we have tried to employ economy regarding causal theory, and hope that the development is fairly readable, for example, to an interested Riemannian geometer, with some basic familiarity with spacetimes.

\emph{Acknowledgements:} The authors would like to express their gratitude to the Mathematical Sciences Research Institute.  We would particularly like to thank the organizers of the 
\href{https://www.msri.org/programs/275}{Mathematical General Relativity Program:} 
\href{https://www.msri.org/people/7885}{James Isenberg}, 
\href{https://www.msri.org/people/27158}{Yvonne Choquet-Bruhat}, 
\href{https://www.msri.org/people/21157}{Piotr Chru{\'s}ciel}, 
\href{https://www.msri.org/people/21226}{Greg Galloway}, 
\href{https://www.msri.org/people/12042}{Gerhard Huisken}, 
\href{https://www.msri.org/people/2699}{Sergiu Klainerman}, 
\href{https://www.msri.org/people/3082}{Igor Rodnianski}, and 
\href{https://www.msri.org/people/2823}{Richard Schoen}, and our colleagues in the cosmology lunch group: 
\href{https://www.msri.org/people/24014}{Lars Andersson},
\href{https://www.msri.org/people/24012}{Hans Ringstr\"om},
\href{https://www.msri.org/people/8272}{Vincent Moncrief}.
We would never have begun this project together had we not both had the opportunity to spend a semester in this wonderfully exciting and welcoming program. We extend additional thanks to Hans Ringstr\"om for serving as the second author's mentor during this program. We are especially indebted to Shing-Tung Yau, Lars Andersson, and also Ralph Howard, for inspiring discussions on the question of convergence in the Lorentzian setting. We would also like to thank Stephanie Alexander, for her interest in the subject of metrizing spacetimes, and her invitations to speak at Urbana, and thank Ming-Liang Cai, Pengzi Miao, Steve (Stacey) Harris, and James Hebda for their vital support. We extend special thanks to Greg Galloway, for his continued support in general, his interest in this project, and his valuable feedback and input on this paper. The authors would also like to thank the referees for their feedback, and numerous valuable comments and suggestions, which have led to various improvements in the present draft. Finally, the second author would like to thank Professor Sormani for the wonderful opportunity and experience of developing this project with her at CUNY.

\section{Lorentzian Background}

For completeness, and to set a few conventions, we begin with a brief review of some basic Lorentzian geometry. For further background, we note the standard references \cite{BEE}, \cite{HE}, \cite{ON},  \cite{Penrose}, \cite{Wald}.

\subsection{Spacetimes}
 
Let $M^{n+1}$ be a smooth manifold of dimension $n + 1$. A Lorentzian metric $g$ on $M$ has signature $(-, +, +, \cdots, +)$, and classifies a vector $X \in TM$ as \emph{timelike}, \emph{null}, or \emph{spacelike} according as $g(X,X)$ is negative, zero, or positive. Timelike and null vectors are referred to collectively as \emph{causal}, and form a double cone in each tangent space. A \emph{spacetime} is a Lorentzian manifold $(M,g)$ which is \emph{time-oriented}, i.e., has a continuous assignment of `future cone' at each point. Hence, every nontrivial causal vector in a spacetime is either \emph{future-pointing} or \emph{past-pointing}. Unless otherwise indicated, $M$ shall henceforth denote a spacetime. Further, we shall take all Lorentzian and Riemannian metrics to be smooth.

\subsection{Causality}

A piecewise smooth curve $\a : I \to M$ is \emph{future timelike (resp. null, causal)} if $\a'$, including all one-sided tangents at any breaks or endpoints, is always future-pointing timelike (resp. null, causal). Past curves are defined time dually. We count trivial (constant) curves as null, and hence causal. However, we will assume that all nontrivial piecewise smooth curves $\a$ are regular, i.e., that $\a$ is parameterized so that all of its tangents are nontrivial, $\a' \ne 0$.  (This can be achieved, for example, by using the arc length parameterization with respect to any choice of Riemannian metric on $M$.) Furthermore, causal curves will implicitly be assumed to be nontrivial where appropriate. The \emph{timelike future} $I^+(S)$ of a subset $S \subset M$ is the set of points $q \in M$ reachable by a future timelike curve from some $p \in S$. The \emph{causal future} $J^+(S)$ is defined similarly using future causal curves, and time dually we have the pasts $I^-(S)$ and $J^-(S)$. For $p, q \in M$, we write $p \le q$ to mean $q \in J^+(p)$, or equivalently $p \in J^-(q)$. We write $p \ll q$ to mean $ q \in I^+(p)$, or equivalently $p \in I^-(q)$. 

\begin{prop} Let $M$ be a spacetime and $S \subset M$ any subset. Then $I^+(S)$ is open, and $S \subset J^+(S) \subset \overline{I^+(S)}$, and similarly for the pasts.
\end{prop}

\subsection{Time Functions}

Following \cite{BEE}, we will say $\tau : M \to \field{R}$ is a \emph{generalized time function} if $\tau$ is strictly increasing along all nontrivial future-directed causal curves. If further $\tau$ is continuous, then $\tau$ is called a \emph{time function}. 

Given any function $f : M \to \field{R}$, the gradient of $f$ at $p \in M$ is defined as usual by $g(\nabla f, X) = Xf$ for all $X \in T_pM$. We note however that if $\nabla f$ is timelike, then $f$ increases in the $- \nabla f$ direction. As in \cite{BStemp}, we will say $\tau : M \to \field{R}$ is a \emph{temporal function} if it is smooth, with past-pointing timelike gradient. It is easy to see that such functions are necessarily time functions. Conversely, given a smooth time function, its gradient is necessarily past-pointing \emph{causal}, but may fail to be everywhere timelike. 

For $p, q \in M$, we will refer to $I^+(p) \cap I^-(q)$ as a `timelike diamond', and similarly to $J^+(p) \cap J^-(q)$ as a `causal diamond'. A spacetime is \emph{strongly causal} if each point admits arbitrarily small timelike diamond neighborhoods. A spacetime is \emph{past-distinguishing} if $I^-(p) = I^-(q)$ implies $p = q$. Any strongly causal spacetime is necessarily past-distinguishing. We note the following, (\cite{BEE}, \cite{MS}):

\begin{prop} Any past-distinguishing spacetime, (and hence any strongly causal spacetime), admits a generalized time function.
\end{prop}

Furthermore, we recall roughly that a spacetime $(M,g)$ is \emph{causal} if it contains no closed causal curves, and \emph{stably causal} if this property persists under small perturbations of $g$. Any stably causal spacetime is necessarily strongly causal. We note the following fundamental Lorentzian result: 

\begin{thm} [\cite{HE}, \cite{BStemp}, \cite{difftime}] \label{tempiftime} 
Let $M$ be a spacetime. The following are equivalent: 
\ben
\item [(a)] $M$ admits a time function.

\item [(b)] $M$ admits a temporal function.

\item [(c)] $M$ is stably causal, (and hence strongly causal).
\een 
\end{thm}

\subsection{Local Causality}

As on a Riemannian manifold, we say an open set $U \subset M$ is \emph{(geodesically) convex} if it is an exponential normal neighborhood of each of its points. Hence, if $U$ is convex, then any two points $x, y \in U$ are connected by a unique geodesic in $U$. Each point in a spacetime admits arbitrarily small convex neighborhoods. 

The Lorentzian arc length functional of a spacetime $(M,g)$ is defined on the space of causal curves $\a : [a,b] \to M$ by $L_g(\a) := \int_a^b \sqrt{|g(\a', \a')|}ds$. We have the following (\cite{ON}):

\begin{prop}  [] \label{localcausality} Let $U$ be a convex neighborhood in a spacetime $M$. For each $p,q \in U$, let $\g_{pq}$ denote the unique geodesic in $U$ from $p$ to $q$.
\ben
\item [(1)] If there is a timelike (resp. causal) curve in $U$ from $p$ to $q$, then $\g_{pq}$ is timelike (resp. causal). 
\item [(2)] If $\g_{pq}$ is timelike, then $L_g(\g_{pq}) \ge L_g(\a)$, for all causal curves $\a$ in $U$ joining $p$ to $q$. Moreover, the inequality is strict unless, when suitably paramterized, $\a = \g_{pq}$.
\een
\end{prop}

\subsection{Lorentzian Distance}

Note that by Proposition \ref{localcausality}, the causal geodesics of a spacetime are locally length-\emph{maximing}. The \emph{Lorentzian distance function} of a spacetime $(M,g)$ is defined by
$$d_g(p,q) := \sup \, \{ \,L_g(\a) : \a \mathrm{\; is \; future \; causal \; from \; } p \mathrm{\; to \; } q \, \},$$
where the supremum is taken to be 0 when there are no such curves. Hence, $d_g(p,q) = 0$ whenever $p \not \le q$. Note that, despite its name, Lorentzian `distance' is not a true distance function, as it fails to be definite, symmetric, and only satisfies the following `reverse triangle inequality':

\begin{prop} [Reverse Triangle Inequality]\label{revtriangle} 
\be
x \le y \le z \; \; \Longrightarrow \; \; d_g(x,y) + d_g(y,z) \le d_g(x,z) \nonumber
\ee
\end{prop}

\section{Null Distance}

For the remainder, $\tau$ will denote a generalized time function. Hence, we assume $\tau : M \to \field{R}$ increases strictly along future causal curves, but do not require $\tau$ to be continuous. 

\subsection{Definitions} \label{secDef}

In a spacetime, a `causal curve' is either future or past directed. We begin by extending this class of curves as follows:

\begin{Def} By a \emph{piecewise causal curve} $\b : [a, b] \to M$ we mean that there is a partition $a = s_0 < s_1 < ... < s_{k-1} < s_k = b$ such that the restriction of $\b$ to $[s_{i-1}, s_i]$ is either a smooth future or past causal curve. In particular, $\b$ is allowed to move backwards and forwards in time. We will sometimes use the notation $\b_i := \b|_{[s_{i-1}, s_i]}$, and $\b = \b_1 \cdot \b_2 \cdots \b_k$, where the dot denotes the natural concatenation of curves, (here from left to right). 
\end{Def}
\begin{figure} [h]
\begin{center}
\includegraphics[width=9cm]{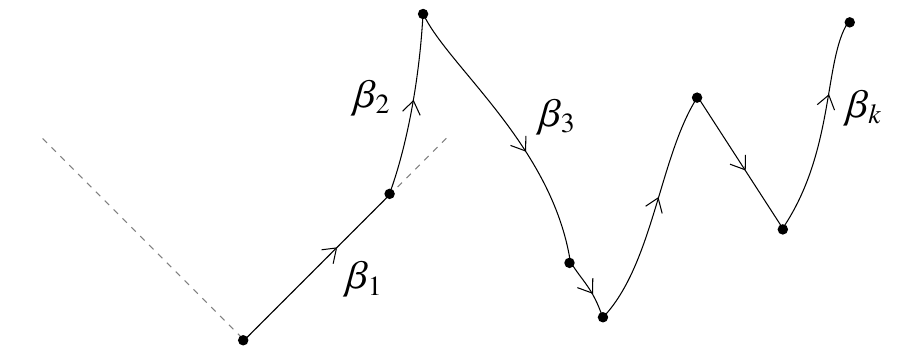}
\caption[...]{A piecewise causal curve $\b = \b_1 \cdot \b_2 \cdots \b_k$. \label{pwc}} 
\end{center}
\end{figure}

\begin{Def} \label{nulldistdef}Let $\tau : M \to \field{R}$ be a generalized time function. 
\ben
\item [(1)] For $\b$ piecewise causal with breaks $a = s_0 < ... < s_k = b$, set $x_i = \b(s_i)$. We define the \emph{null length} of $\b$ by
\be
\hat{L}_\tau(\b) := \sum_{i=1}^k |\tau(x_i) - \tau(x_{i-1})| \nonumber
\ee
\item [(2)] For any $p, q \in M$, we define the \emph{null distance}
\be
\dhat(p,q) := \inf \, \{ \, \hat{L}_\tau(\b) : \, \b \mathrm{\, \; piecewise \; causal \; from \;} p \mathrm{\; to\; } q \,\} \nonumber
\ee
\een
\end{Def}

\vspace{.5pc}
\subsection{Two Model Examples}

Before moving on to general properties, we begin with two key examples on \emph{Minkowski space}, $\field{M}^{n+1} = ( \field{R} \times \field{R}^n, \; -dt^2 + d(x^1)^2 + \cdots + d(x^n)^2 )$.
Let $t$ be the standard time coordinate on $\field{M}^{n+1}$. For $p \in \field{M}^{n+1}$, we will write $p = (t_p, p_S)$, where $t_p = t(p)$, and $p_S \in \field{R}^n$ is the `spatial' part of $p$. Let $||x||$ denote the Euclidean norm of $x = (x^1, ..., x^n)$.
\begin{prop} \label{Mink} Consider $\tau = t$ on Minkowski space, $\field{M}^{n+1}$. Then we have:

\ben
\item [(1)] $\hat{d}_t$ is a definite, translation invariant metric on $\field{M}^{n+1}$, with
\[
\hat{d}_t(p,q) = \left\{
        \begin{array}{ll}
            |t(q) - t(p)| & \; \; \; q \in J^\pm(p) \\ [.7em]
           \, || \, q_S - p_S\,|| & \; \; \; q \not \in J^\pm(p) \\
        \end{array}
    \right.
\]
\item [(2)] $\hat{d}_t$ encodes the causality of $\field{M}^{n+1}$ via:
$$p \le q \; \; \Longleftrightarrow \; \; \hat{d}_t(p,q) = t(q) - t(p)$$
\item [(3)] The $\hat{d}_t$-spheres are coordinate cylinders. For example, the null distance sphere of radius $r$ from the origin is the cylinder given by the equation: $\; \max\{|t|, ||x||\} = r$
\een
\begin{figure} [h]
\begin{center}
\includegraphics[width=5cm]{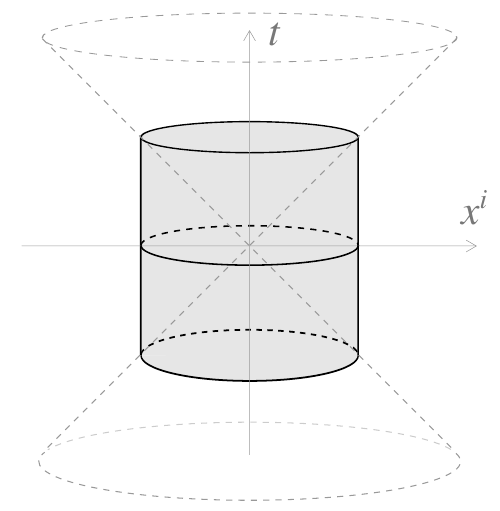}
\caption[...]{The $\hat{d}_t$ spheres on $\field{M}^{n+1}$ are coordinate cylinders. \label{nullball}} 
\end{center}
\end{figure}
\end{prop}

\begin{proof} It is clear that $\hat{d}_t$ is translation invariant. To prove the formula in (1), we may thus take $p = p_0$ to be the origin, and $q = (t,x^1, ..., x^n)$, with $t \ge 0$. If $p_0 \le q$, it follows that $\hat{d}_t(p_0,q) = t$, (see Lemma \ref{dhatonacone} below). On the other hand, if $p_0 \not \le q$, then letting $r = ||x||$, the line segment from the origin to $(0, x^1, ..., x^n)$ lifts to a piecewise null curve $\eta$ with one break as in Figure \ref{minkproof}.

\vspace{.5pc}
\begin{figure} [h]
\begin{center}
\includegraphics[width=5.5cm]{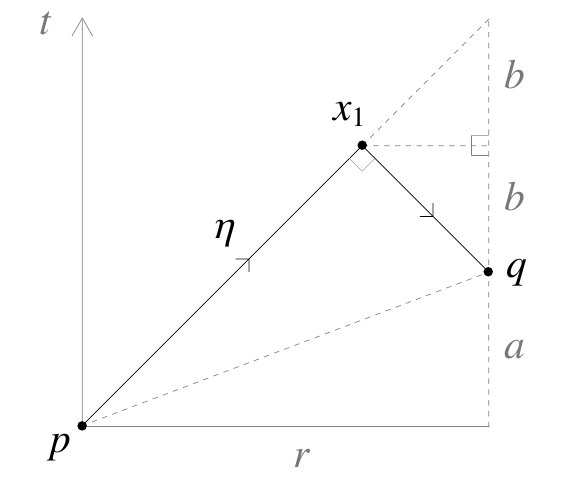}
\caption[...]{The piecewise null curve $\eta$ realizes $\hat{d}_t(p,q)$. \label{minkproof}} 
\end{center}
\end{figure}

The null length of $\eta$ is given by
\be
\hat{L}_t(\eta) = |t(x_1) - t(p)| + |t(q) - t(x_1)| = (a+b) + (b) = r \nonumber
\ee
and hence $\hat{d}_t(p_0,q) \le r = ||x||$. Now let $\b = \b_1 \cdot \b_2 \cdots \b_k$ be any piecewise causal curve from $p_0$ to $q$. Note that either $\b_i$ or $-\b_i$ is future causal. We may thus parameterize either $\b_i$ or its reverse as $(\pm) \b_i(t) = (t, \s_i(t)) : [t_i, t_i + \delta_i] \to M$, with $\delta_i \ge 0$, (where $(\pm) \b_i$ here denotes whichever is future-directed, either $\b_i$ or $-\b_i$). Since $\b_i$ is causal, we have $0 \ge g(\b_i', \b_i') = -1 + h(\s_i', \s_i')$, where $h$ is the standard Riemannian metric on $\{t = 0\} \approx \field{R}^n$. This gives
\be
\hat{L}_t(\b) \, = \; \sum_{i = 1}^k \delta_i \; \ge \;  \sum_{i = 1}^k \int_{t_i}^{t_i + \delta_i}\sqrt{h(\s_i', \s_i')}dt \, = \, \sum_{i = 1}^k L_h(\s_i) \, = \, L_h(\s) \, \ge \, r \nonumber
\ee
where $\s = \s_1 \cdots \s_k$ is the projection of $\b$ onto the $\{t=0\}$ slice. Taking the infimum over all such $\b$ gives $\hat{d}_t(p_0,q) \ge r$. Hence we have $\hat{d}_t(p_0,q) = r = ||x||$. This gives the formula in (1), from which (2) and (3) follow.
\end{proof}

\pagebreak
In contrast to Proposition \ref{Mink}, we have:

\begin{prop} \label{tcubed} Consider $\tau = t^3$ on Minkowski. Then we have the following. 
\ben
\item [(1)] $\hat{d}_\tau$ fails to be definite. In particular, for any two points $p, q$ in the $\{t = 0\}$ slice, we have $\hat{d}_\tau(p,q) = 0$.
\item [(2)] $\hat{d}_\tau$ fails to encode the causal structure. In particular, for any two points $p = (t_p, p_S)$ and $q = (t_q, q_S)$, with $t_p < 0 < t_q$, we have $\hat{d}_\tau(p,q) = \tau(q) - \tau(p)$.
\een
\end{prop}

\begin{proof} It suffices to consider the Minkowksi plane, $\field{M}^{1+1} = \{(t,x)\}$. (1) Let $\eta_k$ be the piecewise null curve joining $p = (t,0)$ to $q = (t,\ell)$, which bounces between the $t$ and $t + \ell/2k$ slices as in Figure \ref{tcubedproof}.
\begin{figure} [h]
\begin{center}
\includegraphics[width=10cm]{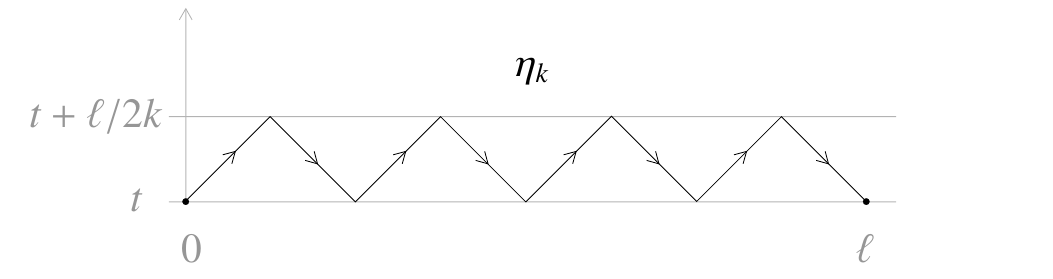}
\caption[...]{\label{tcubedproof}} 
\end{center}
\end{figure}
The null length of $\eta_k$ with respect to $\tau = t^3$ is given by:
\be
\hat{L}_\tau(\eta_k) =  \bigg((t+\ell/2k)^3 - t^3\bigg) \cdot 2k \; = \;  \dfrac{(t+\ell/2k)^3 - t^3}{\frac{\ell}{2k}}\cdot \ell \nonumber
\ee
Hence, we have 
\be
\hat{d}_\tau(p,q) \le \lim_{k \to \8} \hat{L}_\tau(\eta_k) = 3t^2 \cdot \ell \nonumber
\ee
This gives $\hat{d}_\tau(p,q) = 0$, for all $p, q \in \{t = 0\}$.
For (2), let $p_0 := (0, p_s)$ and $q_0 : = (0,q_S)$. Then we have $\hat{d}_\tau(p_0, q_0) = 0$, by (1). Since null distance satisfies the triangle inequality, and with the help of Lemma \ref{indefslice} below, this gives:
\be
\tau(q) - \tau(p) \le \hat{d}_\tau(p,q) \le \hat{d}_\tau(p,p_0) + \hat{d}_\tau(q_0,q) = - \tau(p) + \tau(q) \nonumber
\ee
\end{proof}

\subsection{Basic Properties}

\vspace{.5pc}
\begin{lem} \label{pwcdensity} There is a piecewise causal curve joining any two points $p, q \in M$. 
\end{lem}

\begin{proof} (We assume here, of course, that $M$ is connected.) Fix any path $\a$ from $p$ to $q$. By compactness, there are finitely many points $\{x_2, x_4, ... , x_{2m}\}$, (indexed for convenience below), whose timelike pasts cover the image of $\a$. By connectedness, and by reordering if necessary, we may assume $p \in I^-(x_2)$, $q \in I^-(x_{2m})$, and $I^-(x_{2i}) \cap I^-(x_{2i+2}) \ne \emptyset$. Since $p \in I^-(x_2)$, there is a future timelike curve $\b_1$ from $p$ to $x_2$. Similarly, for $i < m$, since there is a point $x_{2i+1} \in I^-(x_{2i}) \cap I^-(x_{2i+2})$, there is a past timelike curve $\b_{2i}$ from $x_{2i}$ to $x_{2i+1}$, and a future timelike curve $\b_{2i+1}$ from $x_{2i+1}$ to $x_{2i+2}$. Finally, since $q \in I^-(x_{2m})$, there is a past timelike curve $\b_{2m}$ from $x_{2m}$ to $q$. Then $\b = \b_1 \cdots \b_{2m}$ is a piecewise causal (timelike) curve from $p$ to $q$. (See also Remark \ref{onlypwn} below.)
\end{proof}

\vspace{.5pc}
\begin{lem} \label{Lhat} Let $\tau$ be a generalized time function, and $\b$ a piecewise causal curve from $p$ to $q$. 
\ben
\item [(1)] $\hat{L}_\tau(\b) = 0 \; \Longleftrightarrow \; \b$ is trivial.
\item [(2)] $\hat{L}_\tau(\b) = \tau(q) - \tau(p) \; \Longleftrightarrow \; \b$ is future causal.
\item [(3)] In general, we have
\be
\hat{L}_\tau(\b) \ge \max_{y \in \b} \tau(y) - \min_{x \in \b} \tau(x) \ge |\tau(q) - \tau(p)| \nonumber
\ee
\item [(4)] If $\tau$ is differentiable along $\b : [a,b] \to M$, then $\hat{L}_\tau(\b) = \int_a^b |(\tau \circ \b)'| \;ds $.
\een
\end{lem}

\begin{proof} These are all straightforward from the definitions. (3) follows from the triangle inequality and the fact that the extrema of $\tau \circ \b$ must occur at break points of $\b$.
\end{proof}

\vspace{.5pc}
\begin{rem} \label{onlypwn} Fix any piecewise causal curve $\b$, from $p$ to $q$. It follows from Proposition 3.3 in \cite{FSwavetype}, that there is a broken null geodesic $\g$ from $p$ to $q$, which satisfies $\hat{L}_\tau(\g) = \hat{L}_\tau(\b)$, for any choice of generalized time function $\tau$. Consequently, we note that the null distance function $\hat{d}_\tau$ may just as well be constructed using only piecewise \emph{null} (geodesic) curves. Nonetheless, it is more convenient to use the full family of all piecewise causal curves. (But see Lemma \ref{minimizers} below.)
\end{rem}

\vspace{.5pc}
\begin{lem} \label{dhatpseudo} For any generalized time function $\tau$, the induced null distance $\hat{d}_\tau$ is a pseudometric on $M$, satisfying:
\ben
\item [(1)] $\hat{d}_\tau : M \times M \to [0, \8)$
\item [(2)] $\hat{d}_\tau(p,q) = \hat{d}_\tau(q,p)$
\item [(3)] $\hat{d}_\tau(p,q) \le \hat{d}_\tau(p,z) + \hat{d}_\tau(z,q)$
\item [(4)] $\hat{d}_\tau(p,p) =0$
\een
\end{lem}

\vspace{.5pc}
Note that Lemma \ref{pwcdensity} ensures that $\hat{d}_\tau$ is finite. The rest of Lemma \ref{dhatpseudo} follows easily from the definition. Because the set of causal curves on a spacetime is a conformal invariant, the following is immediate:

\begin{prop} \label{conformal} Let $(M,g)$ be a spacetime. Fixing any generalized time function $\tau$ on $M$, the induced null distance $\hat{d}_\tau$ is a pseudometric on $M$, which is invariant under conformal changes of $g$.
\end{prop}

\vspace{.5pc}
Note that part (3) of Lemma \ref{Lhat} gives the following:

\begin{lem} \label{indefslice} For any generalized time function $\tau$ on $M$, and any $p, q \in M$,
\be
\hat{d}_\tau(p,q) \ge |\tau(q) - \tau(p)| \nonumber
\ee
Consequently, definiteness can only fail for points in the same  time slice:
\be
\hat{d}_\tau(p,q) =0 \implies \tau(p) = \tau(q) \nonumber
\ee
\end{lem}

\vspace{.5pc}
\begin{lem} \label{dhatonacone} Null distance satisfies the following causality property:
\be
p \le q \; \Longrightarrow \; \hat{d}_\tau(p,q) = \tau(q) - \tau(p) \nonumber
\ee
\end{lem}

\begin{proof} Let $\b$ be any future causal curve from $p$ to $q$. Then $\hat{L}_\tau(\b) = \tau(q) - \tau(p)$. Hence, $\hat{d}_\tau(p,q) \le \tau(q) - \tau(p)$. Since the reverse inequality holds in general, by Lemma \ref{indefslice}, we have $\hat{d}_\tau(p,q) = \tau(q) - \tau(p)$.
\end{proof}

\vspace{.5pc}
Lemma \ref{dhatonacone} shows that null distance has a simple formula for causally related points. From the definition, it is tempting to expect naively that the converse of this should hold as well. Proposition \ref{tcubed} illustrates that this is not true in general. However, the converse does indeed hold for $\tau = t$ on Minkowski, as shown in Proposition \ref{Mink}, and this is further extended to more general warped product models in Theorem \ref{warpedthm} below. While the definiteness of null distance is addressed in Section \ref{secdefinite}, the ability of null distance to encode causality via some (local) form of \eqref{causalityencoded} is a question which remains to be explored in future work. For now, we note that this property is `stronger' than definiteness, as follows:

\begin{lem} \label{definiteifcausal} Let $M$ be a spacetime with generalized time function $\tau$ such that:
\be
p \le q \; \Longleftrightarrow \; \hat{d}_\tau(p,q) = \tau(q) - \tau(p) \nonumber
\ee
Then $\hat{d}_\tau$ is definite.
\end{lem}

\begin{proof} Let $p, q \in M$, with $\hat{d}_\tau(p,q) = 0$. By Lemma \ref{indefslice}, we have $\tau(p) = \tau(q)$, and hence by hypothesis, we have $p \le q$. But there can not be any nontrivial future causal curves from $p$ to $q$, otherwise $\tau(q) > \tau(p)$. Hence, $p = q$.
\end{proof}

\subsection{Topology}

\vspace{.5pc}
\begin{lem} \label{bdddiamonds} Let $\tau$ be a generalized time function. 
\ben
\item [(1)] $\tau$ is bounded on diamonds: $p \le x \le q \; \Longrightarrow \; \tau(p) \le \tau(x) \le \tau(q) $
\item [(2)] $\hat{d}_\tau$ is bounded on diamonds: $p \le x,y \le q \; \Longrightarrow \; \hat{d}_\tau(x,y) \le 2(\tau(q) - \tau(p)) $
\een
\end{lem}

\begin{proof} (1) follows directly from the definition. (2) Let $\b$ be a piecewise causal curve consisting of a future causal curve from $x$ to $q$, followed by a past causal curve from $q$ to $y$. Then as in (1) we have:
\be
\hat{L}_\tau(\b) = |\tau(q)-\tau(x)| + |\tau(y)-\tau(q)| \le 2(\tau(q) - \tau(p)) \nonumber
\ee
\end{proof}

\begin{prop} \label{dhatcont} $\hat{d}_\tau$ is continuous on $M \times M$ iff $\tau$ is continuous on $M$. 
\end{prop}

\begin{proof} Suppose first that $\tau$ is continuous. Fix $x, y \in M$. Let $\a_x$ be a future timelike curve through $x = \a_x(0)$, and $\a_y$ a future timelike curve through $\a_y(0) = y$. Let $\delta > 0$ small enough so that $\a_x$ and $\a_y$ are defined on $[-\delta, \delta]$. Then for all $x' \in I^+(\a_x(-\delta)) \cap I^-(\a_x(\delta))$ and $y' \in I^+(\a_y(-\delta)) \cap I^-(\a_y(\delta))$, we have:
\begin{eqnarray}
|\hat{d}_\tau(x, y) - \hat{d}_\tau(x',y')| & \le & \hat{d}_\tau(x,x') + \hat{d}_\tau(y,y')     \nonumber \\[1em]
   & \le & 2 \bigg(\tau(\a_x(\delta)) - \tau(\a_x(-\delta))\bigg) + 2 \bigg(\tau(\a_y(\delta)) - \tau(\a_y(-\delta))\bigg) \nonumber
   \end{eqnarray}
\begin{figure} 
\begin{center}
\includegraphics[width=7.5cm]{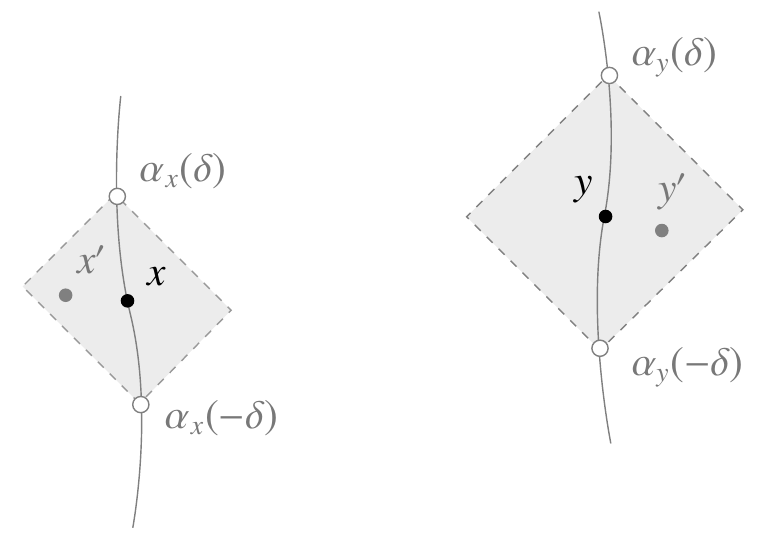}
\caption[...]{Establishing the continuity of $\hat{d}_\tau$. \label{continuity}} 
\end{center}
\end{figure}
Since $\tau$ is continuous, the right hand side goes to 0, as $\delta \to 0$. This shows that $\hat{d}_\tau$ is continuous. Now suppose conversely that $\hat{d}_\tau$ is continuous. Fix any $q_0 \in M$, and let $p_0 \in I^-(q_0)$. Then by Lemma \ref{dhatonacone}, for all $x \in I^+(p_0)$, we have
$$\tau(x)  = \hat{d}_\tau(p_0, x) + \tau(p_0)$$
and hence $\tau$ is continuous at $q_0$. Since $q_0$ was arbitrary, $\tau$ is continuous on $M$.
\end{proof}

\vspace{.5pc}
As a pseudometric, $\hat{d}_\tau$ induces a topology on $M$ generated by the `\emph{open null balls}', $\hat{B}_r(x) := \{ y \in M : \hat{d}_\tau(x,y) < r\}$.

\begin{prop} \label{topology} The topology induced by $\hat{d}_\tau$ coincides with the manifold topology iff $\tau$ is continuous and $\hat{d}_\tau$ is definite.
\end{prop}

\begin{proof} By Proposition \ref{dhatcont}, $\tau$ is continuous iff $\hat{d}_\tau$ is. Suppose first that the topology induced by $\hat{d}_\tau$ coincides with the manifold topology. Then $\hat{d}_\tau$, and hence $\tau$ is continuous. Moreover, since these topologies are Haudorff, $\hat{d}_\tau$ is necessarily definite. Conversely, suppose that $\tau$, and hence $\hat{d}_\tau$ are continuous, and that $\hat{d}_\tau$ is definite. Continuity of $\hat{d}_\tau$ is equivalent to its topology being coarser than the manifold topology. In other words, the `open null balls' are open in the manifold topology. Let $U$ be open in the manifold topology. Fix any $x_0 \in U$. Let $h$ be any Riemannian metric on $M$, $d_h$ its Riemannian distance function, and let $\e > 0$ so that $B = B^h_\e(x_0) := \{y \in M : d_h(x_0,y) < \e\} \subset \subset U$. Because $\hat{d}_\tau$ is definite and continuous, and $\d B$ compact, we have
\be
\e_0 : = \min_{z \in \d B} \hat{d}_\tau(x_0,z) > 0 \nonumber
\ee
Fix any $y_0 \not \in B$. Let $\b : [a, b] \to M$ be a piecewise causal curve from $x_0 = \b(a)$ to $y_0 = \b(b)$. Let $z_0 \in \d B$ be the first point at which $\b$ meets $\d B$. Let $\b_0 \subset \overline{B}$ denote the initial (shortest) portion of $\b$ which runs from $x_0$ to $z_0$. Then we have
\be
\hat{L}_\tau(\b) = \hat{L}_\tau(\b_0)\ge \e_0 \nonumber
\ee
Taking the infimum over all such curves $\b$, we have shown 
\be
y_0 \not \in B \; \; \Longrightarrow \; \; \hat{d}_\tau(x_0, y_0) \ge \e_0 \nonumber 
\ee
In other words, $\hat{B}_{\e_0}(x_0) \subset B \subset U$. Because $x_0 \in U$ was arbitrary, this shows $U$ is open in the null topology. Hence, the topologies coincide.
\end{proof}

\subsection{Scaling}

It follows from the definition that $\hat{d}_\tau$ scales with $\tau$:

\begin{lem} \label{scaling} Let $\tau$ be a generalized time function on a spacetime $(M,g)$. For any positive constant $\lambda >0$, we have $\hat{d}_{\lambda \tau} = \lambda \hat{d}_\tau$.
\end{lem}

Furthermore, we note:

\begin{lem} \label{unique} For any generalized time functions $\tau_1$, $\tau_2$ on $M$, we have:
\be
\hat{d}_{\tau_1} = \hat{d}_{\tau_2}  \; \Longleftrightarrow \; \tau_2 = \tau_1 + C \nonumber
\ee
\end{lem}

\begin{proof} If $\tau_2 = \tau_1 + C$, then it is clear from the definition of null distance that $\hat{d}_{\tau_1} = \hat{d}_{\tau_2}$. Conversely, suppose that the null distances agree. Fix $p_1, p_2 \in M$, and let $\b : [a,b] \to M$ be a piecewise causal curve from $p_1 = \b(a)$ to $p_2 = \b(b)$, with breaks $a = s_0 < s_1 < \cdots s_k = b$. Let $x_i := \b(s_i)$ and $\b_i := \b|_{[s_{i-1}, s_i]}$. If $\b_i$ is future causal, then by Lemma \ref{dhatonacone}, we have
\be 
\tau_2(x_i) - \tau_2(x_{i -1}) = \hat{d}_{\tau_2}(x_{i -1},x_i) =  \hat{d}_{\tau_1}(x_{i -1},x_i) = \tau_1(x_i) - \tau_1(x_{i -1}) \nonumber
\ee
Hence, $(\tau_2 - \tau_1)(x_i) = (\tau_2 - \tau_1)(x_{i-1})$. Since the exact same equality holds when $\b_i$ is past causal, we get this equality for any pair of breaks $x_i$ and $x_j$, including $x_0 = p_1$ and $x_k = p_2$:
\be
(\tau_2 - \tau_1)(p_1) = (\tau_2 - \tau_1)(p_2) \nonumber
\ee
\end{proof}

\vspace{.5pc}
Combining Lemmas \ref{unique} and \ref{scaling}, we have the following.

\begin{prop} If $\tau_1, \tau_2$ are generalized time functions, and $\lambda \in (0,\infty)$, then:
$$\hat{d}_{\tau_2} = \lambda \hat{d}_{\tau_1}  \; \Longleftrightarrow \; \tau_2 = \lambda \tau_1 + C$$
\end{prop}

\subsection{Minimal Curves}

Let $\b$ be a piecewise causal curve from $p$ to $q$. We say $\b$ is \emph{minimal} if $\hat{d}_\tau(p,q) = \hat{L}_\tau(\b)$. The following is immediate from observations made above:

\begin{cor} \label{causalisminimal} Let $p \le q$. Then any causal curve $\a$ from $p$ to $q$ is minimal, with $\hat{d}_\tau(p,q) = \hat{L}_\tau(\a) = \tau(q) - \tau(p)$. In particular, causal curves are null distance-realizing.
\end{cor}

\vspace{.5pc}
The next result characterizes minimal piecewise causal curves, and helps explain the name `null distance'. Note that in general, a set $S \subset M$ is \emph{achronal} if no two points in $S$ are joined by a timelike curve. 

\begin{lem} \label{minimizers} Suppose that $\b$ is a minimal piecewise causal curve. Then either $\b$ is causal, or $\b$ is an achronal piecewise null geodesic, which changes direction in time at each break.
\end{lem}

\begin{proof} Note that by assumption, $\b : [a,b] \to M$ satisfies $\hat{L}_\tau(\b|_{[u,v]}) = \hat{d}_\tau(\b(u), \b(v))$, for all $[u,v] \subset [a,b]$. Let $a = s_0 < s_1 < ... < s_k = b$ be the breaks of $\b$. If $\b$ is not future or past causal, then we may assume there is a break point $s_i$ at which $\b$ changes from future to past causal. Hence, we have $\b_i = \b|_{[s_{i-1}, s_i]}$ future causal, and $\b_{i+1} = \b|_{[s_i, s_{i+1}]}$ past causal. Suppose that, for some $u_0 \in [s_{i-1}, s_i]$, we have $\b(u_0) \ll \b(s_i)$. Then, for some $v_0 \in (s_i, s_{i+1}]$, we have $\b(u_0) \ll \b(v_0)$. Letting $\a_0$ be a future timelike curve from $\b(u_0)$ to $\b(v_0)$, we have
\begin{eqnarray}
\hat{d}_\tau(\b(u_0), \b(v_0)) \; \le \; \hat{L}_\tau(\a_0)
   & = & \tau(\b(v_0)) - \tau(\b(u_0))  \nonumber\\[.5pc]
   & < & \tau(\b(s_i)) - \tau(\b(u_0)) + \tau(\b(s_i)) - \tau(\b(v_0)) \nonumber\\[.5pc]
   & = & \hat{L}_\tau(\b|_{[u_0,v_0]}) \nonumber
   \end{eqnarray}
But this contradicts our hypotheses. Hence, for all $u_0 \in [s_{i-1}, s_i]$, we have $\b(u_0) \le \b(s_i)$ but $\b(u_0) \not \ll \b(s_i)$. It follows from standard causal theory that $\b_i$ is a future null geodesic. By the same argument, $\b_{i+1}$ is a past null geodesic. To extend to the rest of $\b$, suppose, for example, that $a = s_0 < s_1 < s_i$. If $\b_{i-1}:= \b_{[s_{i-2},s_{i-1}]}$ is past causal, then the argument above shows that $\b_{i-1}$ is a past null geodesic. Suppose then that $\b_{i-1}$ is future causal. Then $\b(s_{i-2}) \le \b(s_i)$, but by an argument as above, $\b(s_{i-2}) \not \ll \b(s_i)$. Hence, $\b_{i-1} \cdot \b_i = \b|_{[s_{i-2}, s_i]}$ is a future null geodesic.  
\end{proof}

\vspace{.5pc}
Strictly speaking, we refer to null \emph{pregeodesics} above. Hence, Lemma \ref{minimizers} may be interpreted as a conformal statement. We note the following consequence:

\begin{cor} \label{timelikenotmin} Fix $p, q \in M$, with $p$ and $q$ not causally related, and let $\b$ be a piecewise causal curve from $p$ to $q$. If $\b$ has a timelike subsegment, then there is a shorter piecewise causal curve $\a$ from $p$ to $q$, i.e., $\hat{L}_\tau(\a) < \hat{L}_\tau(\b)$.
\end{cor}

\subsection{Warped Products}

In this section, we consider warped product spacetimes of the form
$$(M^{n+1},g) = \bigg(I \times S^n, -dt^2 + f^2(t)h \bigg),$$
where $I \subset \field{R}$ is an open interval, $f : I \to (0,\8)$ is a smooth, positive function, and $(S^n,h)$ is a Riemannian manifold. Such spacetimes are also referred to as \emph{Generalized Robertson Walker (GRW) spacetimes}. For $p \in M$, we will write $p = (t_p,p_S)$. When considering $t$ as time function on $M$, we will write $t(p) = t_p$.

\vspace{1pc}
We begin by showing that the null distance induced by $\tau = t$ is definite. 

\begin{lem} \label{dhatwarp} Consider a warped product spacetime as above. Then the null distance function $\hat{d}_t$ induced by $\tau = t$ is definite.
\end{lem}

\begin{proof} Let $p \ne q$. We want to show that $\hat{d}_t(p, q) > 0$. Since $\hat{d}_t(p,q) \ge |t(q) - t(p)|$ by Lemma \ref{indefslice}, it suffices to consider the case $t(p) = t(q) = t_0 \in I$. Fix $\delta > 0$ with $[t_0 - \delta, t_0 + \delta] \subset I$, and let $\b = \b_1 \cdot \b_2 \cdots \b_k$ be a piecewise causal curve from $p$ to $q$, with $\b \subset t^{-1}([t_0-\delta, t_0 + \delta])$. Note again that either $\b_i$ or $- \b_i$ is future causal, and hence we may parameterize either $\b_i$ or its reverse as a future causal curve, $(\pm) \b_i (t) = (t, \s_i(t)) : [t_i, t_i + \delta_i] \to M$, with $\s_i$ a smooth curve in $S$, and $[t_i, t_i + \delta_i] \subset [t_0 - \delta, t_0 + \delta]$. Since $\b_i$ is causal, we have 

$$g(\b_i'(t), \b_i'(t)) \le 0 \; \; \Longleftrightarrow \; \; h(\s_i'(t), \s_i'(t)) \le \frac{1}{f^2(t)}$$
Letting $m_0 := \min \{f(t) : t \in [t_0 + \delta, t_0 + \delta]\} > 0$, this gives
$$L_h(\s_i) = \int_{t_i}^{t_i + \delta_i} \sqrt{h(\s_i', \s_i')} \le \int_{t_i}^{t_i + \delta_i} \frac{1}{f(t)}dt \le  \dfrac{\delta_i}{m_0}$$
Since $\delta_i = \hat{L}_t(\b_i)$, this gives
$$\hat{L}_t(\b_i) \ge m_0 L_h(\s_i)$$
and hence,
$$\hat{L}_t(\b) = \sum_{i=1}^k \hat{L}_t(\b_i) \ge \sum_{i = 1}^k m_0 \, L_h(\s_i) = m_0 \, L_h(\s),$$
where $\s = \s_1 \cdot \s_1 \cdots \s_k$ is a piecewise smooth curve in $S$, joining $p_S$ to $q_S$. Since $p \ne q$, but $t(p) = t(q)$, we have $p_S \ne q_S$, and hence we have 
$$\hat{L}_t(\b) \ge  m_0 \, L_h(\s) \ge m_0 \, d_h(p_S, q_S) > 0$$
This shows that, for any piecewise causal curve $\a$ joining $p$ to $q$, we have:
\begin{displaymath}
   \hat{L}_t(\a) \ge \left\{
     \begin{array}{lr}
       \delta &  \a \not \subset t^{-1}([t_0 - \delta, t_0 + \delta])\\[.5pc]
       m_0 \, d_h(p_S, q_S) & \; \; \a \subset t^{-1}([t_0 - \delta, t_0 + \delta])
     \end{array}
   \right.
\end{displaymath} 
Consequently, we have $\hat{d}_t(p,q) \ge \min \{\delta, m_0 \, d_h(p_S, q_S)\} > 0$. 
\end{proof}

\vspace{.5pc}
We will need the following before advancing to Lemma \ref{dhatwarpiscausal}.

\begin{lem} \label{warpedcones} Consider a warped product spacetime $M = I \times_f S$ as above, with $(S,h)$ complete. We have the following:
$$p \le q \; \; \Longleftrightarrow \; \; d_h(p_S,q_S) \le \int_{t_p}^{t_q}\dfrac{1}{f(t)}dt$$
\end{lem}

\begin{proof} The first part is similar to Lemma \ref{dhatwarp}. If $p \le q$, then there is a future causal curve $\b$ from $p$ to $q$ which can be parameterized with respect to $t$, with $\b(t) = (t, \s(t))$, for $t \in [t_p,t_q]$. That $\b$ is causal means 
$$0 \ge g(\b', \b') = -1 + f^2(t)h(\s',\s')$$
Hence, we have
$$d_h(p_S,q_S) \le L_h(\s) = \int_{t_p}^{t_q}\sqrt{h(\s', \s')}dt \le \int_{t_p}^{t_q} \dfrac{1}{f(t)}dt$$
Suppose on the other hand that $d:= d_h(p_S,q_S) \le \int_{t_p}^{t_q}\frac{1}{f(t)}dt$. Let $\s : [0, d] \to S$ be a minimal, unit-speed geodesic in $S$, from $p_S$ to $q_S$. Consider the function
$$\phi(t) := \int_{t_p}^t \dfrac{1}{f(s)}ds$$
Note that $\phi$ is increasing, and by assumption, there is a $t_d \in [t_p, t_q]$ such that $\phi(t_d) = d$.
Define a curve $\b : [t_p, t_q] \to M$ by
\begin{displaymath}
   \b(t) = \left\{
     \begin{array}{lr}
      (t \; , \; \s(\phi(t))) & \hspace{3pc}  t_p \le t \le t_d\\[.5pc]
        (t \; , \; q_S ) & t_d \le t \le t_q
     \end{array}
   \right.
\end{displaymath} 
It is straightforward to check that $\b$ is a causal curve from $p$ to $q$. 
\end{proof}

\vspace{.5pc}
\begin{lem} \label{dhatwarpiscausal} Consider a warped product spacetime $M = I \times_f S$ as above, with $(S,h)$ complete. Then the null distance induced by $\tau = t$ satisfies:
$$p \le q \; \Longleftrightarrow \; \hat{d}_t(p,q) = t(q) - t(p)$$
\end{lem}

\begin{proof} Recall that `$\Longrightarrow$' is true in general, as in Lemma \ref{dhatonacone}. Suppose then that $\hat{d}_t(p,q) = t(q) - t(p) = t_q - t_p$, for some $p, q \in M$. Fix $\delta > 0$ such that $[t_p - \delta, t_q + \delta] \subset I$, and let $m_0 : = \min \{f(t) : t \in [t_p-\delta, t_q + \delta]\}$. For $\e \in (0,\delta)$, let $\b$ be a piecewise causal curve from $p$ to $q$ with $\hat{L}_t(\b) \le t_q - t_p + \e$. Note that by Lemma \ref{Lhat}, we have $\b \subset t^{-1}([t_p - \delta, t_q + \delta])$. Writing $\b = \b_1 \cdot \b_2 \cdots \b_k$, with $(\pm) \b_i(t) = (t, \s_i(t)) : [t_i, t_i + \delta_i] \to M$, and $\s = \s_1 \cdot \s_2 \cdots \s_k$, as in Lemma \ref{dhatwarp}, we have:
$$L_h(\s) \; = \; \sum_{i = 1}^k L_h(\s_i) \; = \; \sum_{i = 1}^k\int_{t_i}^{t_i+\delta_i}\sqrt{h(\s_i', \s_i')} \; dt \;  \le \;  \sum_{i = 1}^k\int_{t_i}^{t_i+\delta_i}\dfrac{1}{f(t)} \; dt$$
Extracting the integral from $t_p$ to $t_q$ from the last term, what remains is a sum of integrals of $1/f$ over all the `extra time wiggles' $\b$ takes in going from $p$ to $q$. But the sum of the sizes of these extra wiggles is exactly the difference $\hat{L}_t(\b) - (t_q-t_p)$. Bounding $1/f$ by $1/m_0$ in all the `extra integrals', we have:
\begin{eqnarray}
\sum_{i = 1}^k\int_{t_i}^{t_i+\delta_i}\dfrac{1}{f(t)} \; dt  & \le & \int_{t_p}^{t_q}\dfrac{1}{f(t)} \; dt + \bigg(\hat{L}_t(\b) - (t_q - t_p)\bigg)\dfrac{1}{m_0}    \nonumber \\ [.5pc]
   & \le & \int_{t_p}^{t_q}\dfrac{1}{f(t)} \; dt + \dfrac{\e} {m_0} \nonumber 
   \end{eqnarray}
Since $d_h(p_S,q_S) \le L_h(\s)$, putting the above together gives:
$$d_h(p_S,q_S) \le \int_{t_p}^{t_q}\dfrac{1}{f(t)} \; dt + \dfrac{\e} {m_0}$$
Since $\e \in (0,\delta)$ was arbitrary, it follows from Lemma \ref{warpedcones} that $p \le q$.
\end{proof}

\vspace{.5pc}
The following generalizes Proposition \ref{Mink}:

\begin{thm} \label{warpedthm} Consider a warped spacetime,
$$(M,g) = \bigg(I \times S, -dt^2 + f^2(t)h\bigg),$$
with $I \subset \field{R}$ an interval, $f : I \to (0,\8)$ smooth, and $(S,h)$ a complete Riemannian manifold. Let $\tau(t,x) = \phi(t)$ be a smooth time function on $M$, with $\phi' > 0$. Then the induced null distance $\hat{d}_\tau$ is definite, and encodes the causality of $M$ via:
$$p \le q \; \Longleftrightarrow \; \hat{d}_\tau(p,q) = \tau(q) - \tau(p)$$
\end{thm}

\begin{proof} We have $d \tau = \phi'(t) d t$, and letting $F(\tau) := \phi'(\phi^{-1}(\tau)) \ne 0$, 
\be
dt^2 = \frac{1}{F^2(\tau)}d \tau^2 \nonumber
\ee
Hence, we may rewrite the metric as
\be
g = -dt^2 + f^2(t)h = \dfrac{1}{F^2(\tau)}\bigg(- d \tau^2 + \rho^2(\tau)h \bigg) = \frac{1}{F^2(\tau)} \, \widetilde{g} \nonumber
\ee
where $\rho(\tau) = F(\tau)f(\phi^{-1}(\tau))$, and $\widetilde{g} := -d \tau^2 + \rho^2(\tau)h$ is a conformal metric on $J \times S \approx M$, where $J = \phi(I)$. By conformal invariance of null distance, we have $\hat{d}_\tau(g) = \hat{d}_\tau(\, \widetilde{g}\, )$. The conclusion then follows from Lemmas \ref{dhatwarp} and \ref{dhatwarpiscausal} 
\end{proof}

\vspace{.5pc}
The following may be compared with Proposition \ref{tcubed}.

\begin{cor} Consider $\tau = t^3$ on Minkowski. Then the induced null distance $\hat{d}_\tau$ is definite and encodes causality when restricted to either the future half of Minkowski, $\field{M}^{n+1}_+ = \{t > 0\}$, or the past half, $\field{M}^{n+1}_- = \{t < 0\}$.
\end{cor}

\vspace{.5pc}
The results above generalize in various ways. We note first that the completeness of $(S,h)$ is not strictly necessary. This is only used (directly)  in Lemma \ref{warpedcones}, to guarantee that any two points in $S$ are joined by a minimal $h$-geodesic. Hence this latter property, sometimes referred to as \emph{(weak) convexity}, suffices throughout. Furthermore, dropping these completeness/convexity assumptions altogether, we have the following:

\begin{lem} \label{warpedcones2} On any GRW spacetime $M = I \times_f S$ as above, (with $S$ arbitrary), we have the following:

\vspace{-.1pc}
$$\hspace{-12pc} (i) \; \; q \in I^+(p) \; \; \Longleftrightarrow \; \; d_h(p_S,q_S) < \int_{t_p}^{t_q}\dfrac{1}{f(t)}dt$$

\vspace{.5pc}
$$\hspace{-12pc} (ii) \; \; q \in \overline{I^+(p)} \; \; \Longleftrightarrow \; \; d_h(p_S,q_S) \le \int_{t_p}^{t_q}\dfrac{1}{f(t)}dt$$
\end{lem}

\vspace{.1pc}
\begin{proof} (i) A straightforward modification of the proof of Lemma \ref{warpedcones} will suffice. Assuming first that $p \ll q$, strict inequality on the right follows exactly as before. Suppose on the other hand that $d = d_h(p_S,q_S) < \int_{t_p}^{t_q} [f(t)]^{-1}dt$. Then we can find a unit-speed curve $\s_\delta : [0, d + \delta] \to S$, from $p_S$ to $q_S$, with 
$$L_h(\s_\delta) = d + \delta < \int_{t_p}^{t_q}\dfrac{1}{f(t)}dt$$
Letting $\phi(t) := \int_{t_p}^t [f(s)]^{-1}ds$ as before, we have $\phi(t_d) = 0 < d + \delta < \phi(t_q)$, and hence we can find a $t_d \in (t_p,t_q)$, such that $\phi(t_d) = d + \delta$. Defining $\b$ exactly as in Lemma \ref{warpedcones}, then $\b$ is a future causal curve from $p$ to $q$. However, in this case, the second segment, $\b(t)$ for $t \in [t_d,t_q]$, is necessarily nonempty, and timelike. It then follows from basic causal theory that there is a timelike curve from $p$ to $q$. (ii) Both directions follow from (i), (and basic causal theory), by sliding a point $q'$ up and down a short future timelike curve from $q$.
\end{proof}

\vspace{.5pc}
As a consequence, Lemma \ref{dhatwarpiscausal}, and Theorem \ref{warpedthm} generalize accordingly:

\begin{thm} \label{warpedthm2} Let $M = I \times_f S$ be any GRW spacetime as above, (with $S$ arbitrary). If $\tau(t,x) = \phi(t)$ is any smooth time function on $M$, with $\phi' > 0$, then $\hat{d}_\tau$ is definite and we have the following: $$q \in \overline{I^+(p)} \; \Longleftrightarrow \; \hat{d}_\tau(p,q) = \tau(q) - \tau(p)$$
\end{thm}

\vspace{.5pc}
The statement in Theorem \ref{warpedthm2} is pertinent, for one, because removing merely a single point from a spacetime destroys many relationships of the form $p \le q$, while still preserving $q \in \overline{I^+(p)}$. On the other hand, these two statements are equivalent when interpreted locally, and are equivalent outright under many standard global causality conditions. Theorem \ref{warpedthm2} is then, perhaps, the most appropriate form of generalization of the basic case in Proposition \ref{Mink}, and is the statement we would expect null distance to satisfy more broadly, under natural conditions on $\tau$.

\vspace{.5pc}
Finally, we note that the results above remain essentially valid for spacetimes which are (only) conformal to a warped product (equivalently, a product) spacetime as above, including the class of `standard static spacetimes'.

\section{Definiteness} \label{secdefinite}

Recall that $\hat{d}_\tau$ depends on both the spacetime $(M,g)$ itself, as well as the choice of generalized time function $\tau$ on $M$. Indeed, we have seen that on a fixed spacetime, $\hat{d}_\tau$ may be definite for some choices of $\tau$ and indefinite for others. Fixing $M$, two natural questions arise: (i) Given a generalized time function $\tau$ on $M$, is $\hat{d}_\tau$ definite, and can we tell by looking just at $\tau$? (ii) If we are free to choose, can we find \emph{some} $\tau$ for which $\hat{d}_\tau$ is definite? This section will be primarily concerned with the first question. However, we will begin first with some comments addressing both, and in particular the second question. Indeed, note the following:

\begin{lem} \label{dhatvgbar} Suppose $\tau : M \to \field{R}$ is smooth, with past-pointing timelike unit gradient, $||\nabla \tau||_g = 1$. Then $\hat{d}_\tau$ is definite.
\end{lem}
\begin{proof} It is a fairly standard Lorentzian trick (in this case, inspired in part by previous discussions with S.-T. Yau, L. Andersson, and R. Howard), to consider $g^R_\tau (X,Y) : = g(X,Y) + 2g(X, \nabla \tau)g(Y, \nabla \tau)$, which is easily seen to be a smooth Riemannian metric on $M$. (Extend $e_0 = \nabla \tau$ to an orthonormal basis.) Let $L^R_\tau$ and $d^R_\tau$ denote the Riemmanian arc length and distance induced by $g^R_\tau$, respectively. Fix any two points $p, q \in M$, and any piecewise causal curve $\b : [a,b] \to M$, from $p$ to $q$. Hence, $g(\b', \b') \le 0$, and we have:
$$L^R_\tau(\b) = \int_a^b \sqrt{2[g(\nabla \tau, \b')]^2 + g(\b', \b')}dt \le \sqrt{2} \int_a^b  |(\tau \circ \b)'|dt = \sqrt{2} \, \hat{L}_\tau(\b)$$
It follows that $d_\tau^R(p,q) \le \sqrt{2} \hat{d}_\tau(p,q)$, and in particular, that $\hat{d}_\tau$ is definite.
\end{proof}

\vspace{.5pc}

\begin{prop} Suppose that $M$ is stably causal, or equivalently, that $M$ admits a time function $\tau_0$. Then we can always find a (second) time function $\tau$ on $M$ for which $\hat{d}_\tau$ is definite. 
\end{prop}

\begin{proof} Using \cite{BStemp}, for example, we can find a temporal function $\tau$ on $(M,g)$, i.e., $\tau : M \to \field{R}$ smooth, with $\nabla \tau$ everywhere past-pointing timelike. (Because $||\nabla \tau||_g$ is locally bounded away from zero, it is straightforward, if a bit more tedious, to show directly that $\hat{d}_\tau$ is definite. For simplicity, however, we will opt to proceed as follows.) By scaling $\tau$ if necessary, we may suppose that $||\nabla \tau||_g = 1$. The result then follows from Lemma \ref{dhatvgbar}. 
\end{proof}

\subsection{Definiteness and Anti-Lipschitz Conditions}

\vspace{.5pc}
\begin{lem} \label{revLiplem} Let $M$ be a spacetime with generalized time function $\tau$. Suppose that for some neighborhood $U$ in $M$, there is a (definite) distance function $d_U$ on $U$, such that, for all $x, y \in U$, we have:
\be
x \le y \; \; \Longrightarrow \; \; \tau(y) - \tau(x) \ge d_U(x,y) \nonumber
\ee
Then $\hat{d}_\tau$ distinguishes all of the points in $U$. That is, for every $p \in U$, and any $q \in M \setminus \{p\}$, we have $\hat{d}_\tau(p,q) > 0$. In particular, $\hat{d}_\tau$ is definite on $U$.
\end{lem}

\begin{proof} Fix $p \in U$, and $q \ne p$. Let $B$ be any precompact open neighborhood of $p$, with $B \subset \subset U$ and $q \not \in B$. Let $\b$ be a piecewise causal curve from $p$ to $q$. Let $z_0 \in \d B$ be the first point at which $\b$ meets $\d B$. Let $\b_0 \subset \overline{B}$ denote the initial portion of $\b$ which goes up to $z_0$, and let $p = x_0, x_1, ..., x_k = z_0$ denote its breaks. Then, since $\b_0 \subset U$, and by the triangle inequality for $d_U$, 
\be
\hat{L}_\tau(\b) \ge \hat{L}_\tau(\b_0) = \sum_{i=1}^k |\tau(x_i) - \tau(x_{i-1})| \ge  \sum_{i=1}^k  d_U(x_{i-1}, x_i) \ge  d_U(p, z_0) \ge d_U(p,\d B) \nonumber
\ee
Taking the infimum over all such $\b$, we have $\hat{d}_\tau(p,q) \ge d_U(p,\d B) > 0$. 
\end{proof}

\vspace{.5pc}
\begin{Def} \label{revLip} Let $M$ be a spacetime and $f : M \to \field{R}$. Given a subset $U \subset M$, we will say $f$ is \emph{anti-Lipschitz on $U$} if there is a (definite) distance function $d_U$ on $U$ such that, for all $x, y \in U$, we have:
\be
x \le y \; \; \Longrightarrow \; \; f(y) - f(x) \ge d_U(x,y) \nonumber
\ee
We will say $f : M \to \field{R}$ is \emph{locally anti-Lipschitz} if $f$ is anti-Lipschitz on a neighborhood $U$ of each point $p \in M$. 
\end{Def}

Any locally anti-Lipschitz function is necessarily a generalized time function, which by Lemma \ref{revLiplem}, induces a definite null distance function. Moreover, by Lemma \ref{dhatonacone}, this condition is also necessary, and we have the following:

\begin{prop} \label{revLipprop} Let $\tau$ be a generalized time function on $M$. Then $\hat{d}_\tau$ is definite iff $\tau$ is locally anti-Lipschitz.
\end{prop}


Combining Propositions \ref{revLipprop}, \ref{conformal}, and \ref{topology} we have the following. 

\begin{thm} \label{revLipthm} Let $\tau$ be a time function on a spacetime $M$. If $\tau$ is locally anti-Lipschitz, then the induced null distance function $\hat{d}_\tau$ is a definite, conformally invariant metric on $M$, which induces the manifold topology.
\end{thm}

Before moving on, we discuss a few alternate forms of the anti-Lipschitz condition above. Fixing any Riemannian metric $h$ on $M$, let $d_h$ denote the induced Riemannian distance function. Note first that, because all distance functions on a manifold are locally Lipschitz equivalent, the following is immediate:

\begin{lem} \label{revLipRiem} Let $M = (M,g)$ be a spacetime, and fix any Riemannian metric $h$ on $M$. Then the following conditions on a function $f : M \to \field{R}$ are equivalent.
\ben
\item [(1)] For every $p \in M$, there is a neighborhood $U$ of $p$, and a distance function $d_U$ on $U$, such that for all $x, y \in U$,
$$x \le y \; \; \Longrightarrow \; \; f(y) - f(x) \ge  d_U(x,y)$$
\item [(2)] For every $p \in M$, there is a neighborhood $U$ of $p$, and a positive constant $C >0$, such that for all $x, y \in U$,
$$\hspace{1pc} x \le y \; \; \Longrightarrow \; \; f(y) - f(x) \ge C d_h(x,y)$$
\item [(3)] For every $p \in M$, there is a neighborhood $U$ of $p$, and a Riemannian metric $h_U$ on $U$, such that for all $x, y \in U$,
$$x \le y \; \; \Longrightarrow \; \; f(y) - f(x) \ge  d_{h_U}(x,y)$$
\een
\end{lem}

A closely related `anti-Lipschitz' condition is used by Chru{\'s}ciel, Grant, and Minguzzi in \cite{difftime}, and is given roughly by condition (a) in Lemma \ref{antiiffrev} below. That this anti-Lipschitz condition implies the one above is essentially trivial, (modulo basic causal theory). We thank G. Galloway for noting that the less obvious converse should hold as well, and include his argument here. This is based on the following general fact that, in the eyes of a Riemannian metric, a causal curve between two points `can only be so long':

\begin{lem}  \label{bddwiggles} Let $M$ be a spacetime and fix a Riemannian metric $h$ on $M$. Then for each $p \in M$, there is a neighborhood $U$ of $p$, and a positive constant $C > 0$, such that, for each $x, y \in U$, and any causal curve $\a : I \to U$ from $x$ to $y$, we have $L_h(\a) \le C d_h(x,y)$.
\end{lem}

\begin{proof} Because the result is local, it suffices by standard considerations to prove it for Minkowski space, with $h$ equal to the standard Euclidean metric. Let $\a$ be any future causal curve from the $\{t = 0\}$ slice to the $\{t = b\}$ slice. Then we have a parameterization $\a(t) = (t, \vec{x}(t))$, for $0 \le t \le b$. Since $\a$ is causal, we have $||\vec{x} \, '(t)||_h \le 1$. Hence, $L_h(\a) \le 2b \le 2d_h(\a(0), \a(b))$.
\end{proof}

\vspace{.5pc}
\begin{prop}  \label{antiiffrev} Let $M$ be a spacetime, and fix a Riemannian metric $h$ on $M$. The following conditions on a function $f : M \to \field{R}$ are equivalent. 
\ben
\item [(a)] For every $p \in M$, there is a neighborhood $U$ of $p$, and a positive constant $C > 0$, such that for all future causal curves $\a : [a,b] \to U$, we have
\be
f(\a(b)) - f(\a(a)) \ge C L_h(\a) \nonumber
\ee
\item [(b)] For every $p \in M$, there is a neighborhood $U$ of $p$, and a positive constant $C >0$, such that for all $x, y \in U$, with $x \le y$, we have
\be
f(y) - f(x) \ge C d_h(x,y) \nonumber
\ee
\een
\end{prop}

\begin{proof} Because the statements are local, it suffices to assume that $M$ is strongly causal. For (a) $\implies$ (b), fix $p \in M$, and let $U$ and $C$ as in (a). Let $U_0 = I^+(p_-) \cap I^-(p_+)$ be a timelike diamond neighborhood of $p$, with $U_0 \subset \subset U$. Fix any $x, y \in U_0$, with $x \le y$. Since $U_0$ is a diamond, we have $J^+(x) \cap J^-(y) \subset U_0 \subset U$, and hence there is a future causal curve $\a : [a, b] \to U$, from $x = \a(a)$ to $y = \a(b)$. Then we have $f(y) - f(x) \ge CL_h(\a) \ge Cd_h(x,y)$. For (b) $\Longrightarrow$ (a), fix $p \in M$, and let $U$ and $C$ as in (b). Let $U_0 \subset U$ be a neighborhood of $p$, and $C_0 > 0$ a positive constant, both as in Lemma \ref{bddwiggles}. Then for any future causal curve $\a : [a,b] \to U_0$, we have $f(\a(b)) - f(\a(a)) \ge Cd_h(\a(a), \a(b)) \ge CC_0L_h(\a)$.
\end{proof}

\vspace{.5pc}
We note finally that Seifert considered similar `anti-Lipschitz' conditions in \cite{Seifertcosmic}. However, we will not explore these further here.

\subsection{Differentiable Functions}

In this section, we rephrase the anti-Lipschitz condition for differentiable (but not necessarily $C^1$) functions. We begin with a few Lorentzian basics. The following characterizations of past-pointing timelike vectors are easily established.
 
\begin{lem} \label{pasttimelike} Fix $p \in M$, and $T \in T_pM$. The following are equivalent.
\ben
\item [(a)] $T$ is past-pointing timelike.
\item [(b)] For all nontrivial future causal vectors $X \in T_pM$, $g(T,X) >0$.
\item [(c)] Given any Euclidean inner product $h_p$ on $T_pM$, there is a positive constant $C >0$, such that, for all future causal $X \in T_pM$, we have
$$g(T,X) \ge C||X||_{h_p}$$
\een
\end{lem}


\vspace{.5pc}
A smooth function whose gradient is past-pointing timelike is necessarily a time function. Indeed, note that this is true with no regularity assumptions on the gradient:

\begin{lem} \label{timelikefunc} Let $f : M \to \field{R}$ be any function such that $\nabla f$ exists and is past-pointing timelike on all of $M$. (We do not assume $\nabla f$ is continuous.) Then $f$ is a time function.
\end{lem}

\begin{proof} Let $\a : [a, b] \to M$ be a smooth, (nontrivial), future-directed causal curve. (The argument extends easily to $\a$ piecewise smooth.) Since $\nabla f$ is past timelike, and $\a'$ is future causal, with $\a' \ne 0$, (since we assume a regular parametrization), we have $(f \circ \a)' = g(\nabla f, \a') > 0$, by Lemma \ref{pasttimelike}. It follows that $f$ is strictly increasing along $\a$, (e.g., by the Mean Value Theorem).
\end{proof}

\vspace{.5pc}
The following result further reformulates the anti-Lipschitz condition for differentiable functions.

\begin{prop} \label{revLipdiffble} Let $M$ be a spacetime, and let $\tau : M \to \field{R}$ be any function with a well-defined (but possibly discontinuous) gradient $\nabla \tau$ on all of $M$. Letting $h$ be any Riemannian metric on $M$, the following conditions are equivalent:
\ben
\item [(a)] For each $p \in M$, there is a neighborhood $U$ of $p$, and a positive constant $C > 0$, such that for all future causal vectors $X \in TU$, 
\be
g(\nabla \tau, X) \ge C||X||_h \nonumber 
\ee
\item [(b)] For each $p \in M$, there is a neighborhood $U$ of $p$, and a positive constant $C > 0$, such that for all $x, y \in U$, with $x \le y$, we have 
\be
\tau(y) - \tau(x) \ge C d_h(x,y) \nonumber
\ee
\een
If either of the above conditions hold, then $\tau$ is necessarily a (locally anti-Lipschitz) time function, with $\nabla \tau$ past-pointing timelike, and $||\nabla \tau||_g$ locally bounded away from zero.
\end{prop}

\begin{proof} For (a) $\implies$ (b), fix $p \in M$. Let $U$ and $C$ as in (a). Let $\a : [a, b] \to U$ be nontrivial and future causal. Note that (a) implies that $\nabla \tau$ is past-pointing timelike on $M$ by Lemma \ref{pasttimelike}. Hence, $\tau$ is a time function by Lemma \ref{timelikefunc}, and $(\tau \circ \a)$ is strictly increasing on $[a,b]$. Standard analysis for monotonic functions then gives that $(\tau \circ \a)$ is differentiable almost everywhere, and the first inequality below:
$$\tau(\a(b)) - \tau(\a(a)) \ge \int_a^b (\tau \circ \a)'(s)ds = \int_a^b g(\nabla \tau, \a')ds \ge C L_h(\a)$$
This shows that condition `(a)' in Proposition \ref{antiiffrev} holds, and hence gives (b). For (b) $\implies$ (a), fix $p \in M$. Let $U$ and $C >0$ as in (b). Fix any $q \in U$, and any future-pointing timelike vector $X \in T_qU$. For sufficiently small $\e >0$, let $\a : (-\e, \e) \to U$ be an $h$-geodesic through $\a(0) = q$, with $\a'(0) = X/||X||_h$. Since $g(\a'(0), \a'(0)) = g(X,X)/||X||_h^2 < 0$, this remains true on a small interval, and hence $\a$ is timelike near $t = 0$. Furthermore, because $\a$ is an $h$-geodesic, it is $h$-minimal near $t = 0$. In particular, for all sufficiently small $t > 0$, we have $\a(0) \le \a(t)$, and $d_h(\a(0), \a(t)) = t$. Then using the estimate in (b), we get 
\be
g\bigg(\nabla \tau \, , \, \frac{X\,}{\,||X||_h}\bigg) = (\tau \circ \a)'(0) = \lim_{\; t \to 0^+}\dfrac{\tau(\a(t)) - \tau(\a(0))}{t} \ge \lim_{\;t \to 0^+}\dfrac{Ct}{t} = C \nonumber
\ee
Hence, we have shown that, for any future timelike $X \in TU$, we have
$$g(\nabla \tau, X) \ge C||X||_h$$
But by continuity of $g$ and $h$, this extends to all future causal $X \in TU$. Finally, note that either of these conditions imply that $\tau$ is a time function, and that by Lemma \ref{pasttimelike}, (a) implies that $\nabla \tau$ is past-pointing timelike. That $||\nabla \tau||_g$ is locally bounded away from zero follows as in the proof of Lemma \ref{locsteepT} below.
\end{proof}

\subsection{Almost Everywhere Differentiable Functions}

Our main goal in this section is to prove that, if $\tau$ is a time function with gradient almost everywhere, then $\tau$ is locally anti-Lipschitz provided that its gradient (where defined) remains `bounded away from the light cones'. We begin by making this condition precise as follows.

\begin{Def} \label{bddaway} Let $\mathcal{T}$ be a set of timelike vectors. Fix any Riemannian metric $h$ on $M$. We will say $\mathcal{T}$ is \emph{locally bounded away from the light cones} if, for all $p \in M$, there is a neighborhood $U $ of $p$, and a positive constant $C > 0$, such that for every $T \in \mathcal{T} \cap TU$, we have $||T||_g \ge C \max \{1, ||T||_h\}$.
\end{Def}

\vspace{.5pc}
The condition in Definition \ref{bddaway}, in some form or other, is standard throughout Lorentzian geometry. Often $\mathcal{T}$ is a set of unit vectors, in which case the condition can be simplified. Above, we allow for timelike vectors of arbitrary length. Note that the inequality $||T||_g \ge C||T||_h$ keeps $T$ bounded away from the `walls' of the light cone, while $||T||_g \ge C$ keeps $T$ bounded away from the `vertex' of the  cone. (See Figure \ref{bounded}.) We note also that this condition is independent of the choice of Riemannian metric.

\begin{figure} 
\begin{center}
\includegraphics[width=12cm]{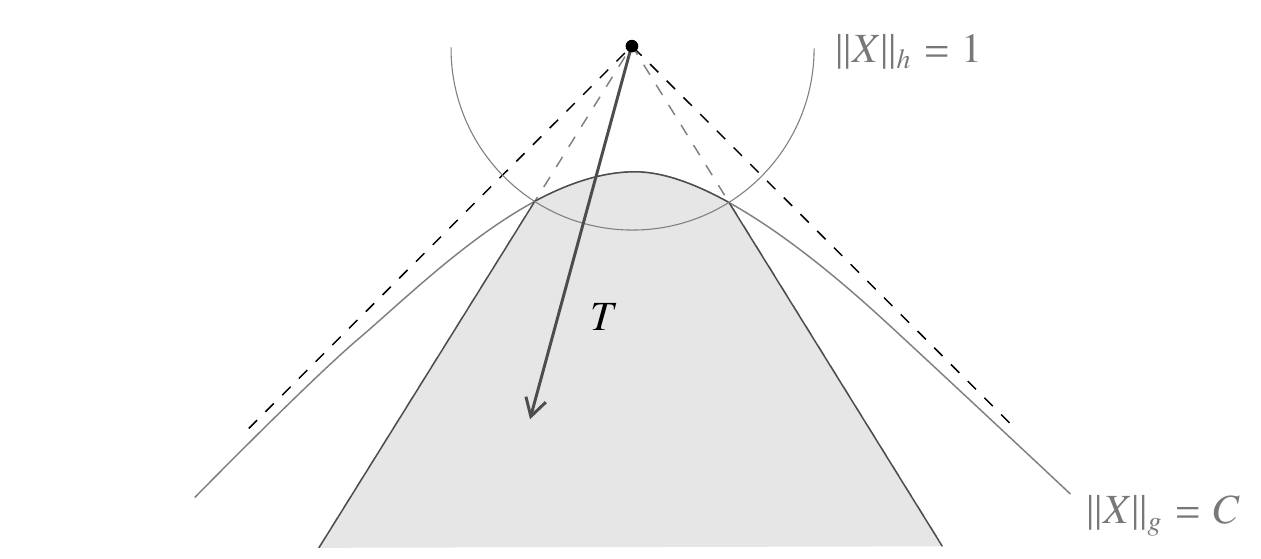}
\caption[...]{A region of past timelike vectors bounded away from the light cone, as in Definition \ref{bddaway}.  \label{bounded}} 
\end{center}
\end{figure}

\vspace{.5pc}
The following basic fact will be used in the proof of Lemma \ref{locsteepT}:

\begin{lem} \label{circVhyper} Let $(M,g)$ be a Lorentzian manifold. Then for any precompact neighborhood $U_0 \subset M$, and any Riemannian metric $h$ on $U_0$, there is a positive constant $C_0 > 0$ such that, for all  vectors $X \in TU_0$, we have $||X||_h \ge C_0||X||_g$.
\end{lem}


\vspace{.5pc}
\begin{lem} \label{locsteepT} Let $\mathcal{T}$ be any collection of vectors on a spacetime $M$. Fix any Riemannian metric $h$ on $M$ and consider the following conditions.
\ben
\item [(a)] The vectors in $\mathcal{T}$ are past-pointing timelike, and locally bounded away from the light cones.
\item [(b)] For each $p \in M$, there is a neighborhood $U$ of $p$, and a positive constant $C > 0$, such that, for all future-pointing causal vectors $X \in TU$, and any $T \in \mathcal{T} \cap TU$, we have $g(T, X) \ge C ||X||_h$.
\item [(c)] The vectors $T \in \mathcal{T}$ are past-pointing timelike, with $||T||_g$ locally bounded away from zero.
\een
In general, we have (a) $\implies$ (b) $\implies$ (c), (though each converse may fail, as in Examples \ref{steepspills} and \ref{tipsnospill} below). If we further assume that the vectors in $\mathcal{T}$ are locally bounded, in particular if $\mathcal{T}$ is a continuous vector field, then all three conditions are equivalent.
\end{lem}

\begin{proof} For (a) $\implies$ (b), we proceed by contradiction. Suppose then that (b) fails at some $p_0 \in M$. Let $U$ and $C$ correspond to $p=p_0$ as in Definition \ref{bddaway}. Hence, considering smaller and smaller $h$-balls around $p_0$ in $U$, we can find a sequence $p_j \in U$, with $p_j \to p_0$, a sequence of future causal vectors $X_j \in T_{p_j}M$, and a sequence $T_j \in \mathcal{T} \cap T_{p_j}M$ such that
$$0 \le g(T_j,X_j) < j^{-1}||X_j||_h$$
Dividing by $||X_j||_h$, we may suppose that $||X_j||_h = 1$, and hence that $X_j \to X_0 \in T_{p_0}M$. Note that $X_0$ is necessarily future causal, and nontrivial, since $||X_0||_h = 1$. Suppose first that $||T_j||_h$ has a bounded subsequence. Then we may suppose that $T_j \to V_0 \in T_{p_0}M$, with $V_0$ past causal, and taking the limit above gives $g(V_0, X_0) = 0$. But this implies that $V_0$ is null, by Lemma \ref{pasttimelike}, and hence that $||T_j||_g \to ||V_0||_g = 0$, contradicting $||T_j||_g \ge C > 0$ from part (a). Suppose then that $||T_j||_h \to \8$. Letting $W_j := T_j/||T_j||_h$, then for all sufficiently large $j$, we have $||T_j||_h \ge 1$, and hence
$$0 \le g(W_j, X_j) < j^{-1}||T_j||^{-1}_h \le j^{-1}$$
Since $||W_j||_h = 1$, we may suppose that $W_j \to W_0 \in T_{p_0}M$. But then again, taking the limit shows that $W_0$ is null, which again contradicts the hypothesis in (a), 
$$0 < C \le ||T_j||_g/||T_j||_h = ||W_j||_g \to ||W_0||_g = 0$$
For (b) $\implies$ (c), first note that Lemma \ref{pasttimelike} implies that each $T \in \mathcal{T}$ is past-pointing timelike. Fix $p \in M$, and let $U$ and $C$ as in (b). Let $U_0$ be any precompact neighborhood of $p$, with $U_0 \subset \subset U$. Then, as in Lemma \ref{circVhyper}, there is a positive constant $C_0 > 0$ such that $||X||_h \ge C_0 ||X||_g$ for all $X \in TU_0$. Then for $q \in U_0$, and any $T \in \mathcal{T} \cap T_qM$, since $X = -T$ is future causal, we have 
$$||T||^2_g = g(T, -T) \ge C||T||_h \ge CC_0||T||_g$$
Hence $||T||_g \ge CC_0 > 0$. Finally, we show (c) $\implies$ (a) under the additional assumption that $T$ is locally bounded. Fix any $p_0 \in M$. Let $U$ be a neighborhood of $p_0$, such that $||T||_g \ge C > 0$, for all $T \in \mathcal{T} \cap TU$. Suppose there is a sequence $p_j \to p_0$, and vectors $T_j \in \mathcal{T} \cap T_{p_j}M$, for which $C \le ||T_j||_g < j^{-1}\max\{1, ||T_j||_h\}$. But since $\mathcal{T}$ is bounded near $p_0$, we have $||T_j||_h \le C_1$ for all sufficiently large $j$, and taking a limit of $0 < C \le ||T_j||_g \le j^{-1}(C_1+1)$ gives a contradiction. 
\end{proof}

\vspace{.5pc}
\begin{cor} \label{revLipsmooth} Let $f : M \to \field{R}$ be any function with continuous gradient $\nabla f$ on all of $M$. Then $f$ is locally anti-Lipschitz iff $\nabla f$ is past-pointing timelike.
\end{cor}

\begin{proof} If $f$ is locally anti-Lipschitz, then $\nabla f$ is past-pointing timelike, by Proposition \ref{revLipdiffble} and Lemma \ref{revLipRiem}. Suppose conversely that $\nabla f$ is past-pointing timelike. Then by Proposition \ref{revLipdiffble} and Lemma \ref{locsteepT}, we need only observe that $||\nabla f||_g$ is locally bounded away from zero, which follows from the continuity of $\nabla f$.
\end{proof}

\vspace{.5pc}
\begin{lem} \label{avoidance} Let $M$ be a spacetime and $Z \subset M$ a set of measure zero. If $p \ll q$, then there is a broken future timelike geodesic $\b : [a,b] \to M$, from $p = \b(a)$ to $q = \b(b)$, such that $\{s \in [a,b] : \b(s) \in Z\}$ has measure zero in $[a,b]$.
\end{lem}

\begin{proof} Let $\a$ be any future timelike curve from $p$ to $q$. By covering the image of $\a$ with a finite number of such neighborhoods if necessary, we may suppose $\a$ maps into a single convex neighborhood $U$. Because $U$ is convex, for each $x, y \in U$, there is a unique geodesic $\g_{xy}$ in $U$ from $x$ to $y$. Moreoever, if there is a timelike curve from $x$ to $y$ within $U$, then the geodesic $\g_{xy}$ is timelike, (see Proposition \ref{localcausality}). Since $\a$ is timelike, we have that $\g_{pq}$ is timelike. Let $z$ be any interior point of $\g_{pq}$, and let $\S \subset U$ be a smooth spacelike hypersurface through $z$, with $\S$ achronal in $U$. Then $I^+(p; U) \cap \S$ is an open neighborhood of $z$ in $\S$, where $I^+(p;U)$ is the relative future of $p$ within $U$ (viewed as a subspacetime), and similarly for $I^-(q;U) \cap \S$. Let $V$ be a neighborhood in $\S$ of $z$ with $V \subset I^+(p; U) \cap I^-(q; U)$. (See Figure \ref{avoid}.) 
\begin{figure} [h]
\begin{center}
\includegraphics[width=7cm]{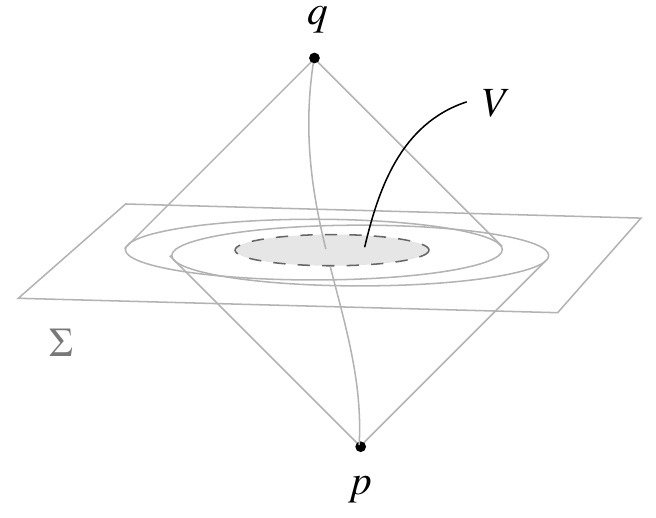}
\caption[The achronal hypersurface $V$ used in the proof of Lemma \ref{avoidance}.] {\label{avoid}}
\end{center}
\end{figure}
Hence, for any $x \in V$, we have that $\g_{px}$ and $\g_{xq}$ are future timelike, and these join to form a broken future timelike geodesic from $p$ to $q$. 

We will focus first on the collection of geodesics from $p$ to $V$, $\{\g_{px} : x \in V\}$. Note that since $U$ is convex, it is a normal neighborhood of $p$, i.e., there is a star-shaped open neighborhood $\widetilde{U}$ of $0$ in $T_pM$ such that the exponential map at $p$ gives a diffeomorphism $\mathrm{exp}_p : \widetilde{U} \to U$. Hence, $\widetilde{V} : = \exp_p^{-1}(V)$ is a smooth hypersurface in $\widetilde{U}$, and since $V$ is achronal, with each $\g_{px}$ parameterized on $[0,1]$, the restriction
$$F=\mathrm{exp}_p : \{tv: \, v\in \widetilde{V},\, t\in (0,1)\}
\to \{\gamma_{px}(t):\, x\in V, \, t\in (0,1)\}$$
is a diffeomorphism between open sets homeomorphic to $V \times (0,1) \approx \widetilde{V} \times (0,1)$. Let $\chi^{\,}_Z : M \to \{0,1\}$ be the indicator function of $Z$. Since $Z$ has measure zero in $M$, Fubini's Theorem gives
$$0 = \int_{\widetilde{V} \times [0,1]} (\chi^{\,}_Z \circ F) \, dx^{n+1} = \int_{\widetilde{V}} \int_0^1 (\chi^{\,}_Z \circ F) \, dt dx^n$$
Since the integrand above is nonnegative, then
for Lebesgue almost every $x \in \widetilde{V}$ we have 
$$0 = \int_0^1 (\chi_Z \circ F)(x,t) \, dt = \int_0^1 \chi_Z \circ\gamma_{px}(t)\, dt$$

\vspace{1pc}
\noindent
Hence, for almost every $x\in V$, we have:
$$\gamma_{px}^{-1}(Z) \textrm{ has Lebesgue measure $0$ in $[0,1]$}$$
To complete the proof, note that by symmetry we similarly have that, for almost every $x \in V$, $\gamma_{xq}^{-1}(Z)$ has measure zero. Hence, for almost every $x \in V$, the concatenation $\b = \g_{px} \cdot \g_{xq}$ is a broken future timelike geodesic from $p$ to $q$, such that $\b^{-1}(Z)$ has measure zero.
\end{proof}

\vspace{.5pc}
We are ready to prove the main result in this section:

\begin{thm} \label{bddgradthm} Suppose that $\tau$ is a time function on $M$, such that the gradient vectors $\nabla \tau$ exist almost everywhere, with $\nabla \tau$ timelike and locally bounded away from the light cones. Then $\tau$ is locally anti-Lipschitz, and hence induces a definite null distance function $\hat{d}_\tau$ on $M$.
\end{thm}

\begin{proof} Fix any Riemannian metric $h$ on $M$, and any $p \in M$. Since the gradient field $\nabla \tau$ is bounded away from the light cones, we may use the `(a) $\Longrightarrow$ (b)' part of Lemma \ref{locsteepT} to choose a neighborhood $U$ of $p$, and a positive constant $C>0$, such that, for all future causal vectors $X \in TU$, we have $g(\nabla \tau, X) \ge C||X||_h$, wherever $\nabla \tau$ is defined. Note that since $\tau$ is a time function, $M$ is necessarily strongly causal, by Theorem \ref{tempiftime}. Let $W$ be a timelike diamond neighborhood of $p$ with $W \subset U$. We will show that $\tau$ is anti-Lipschitz on $W$. 

Fix first any $x, y \in W$ with $x \ll y$. Note that since $W$ is a diamond, any future causal curve from $x$ to $y$ lies entirely within $W$. Using Lemma \ref{avoidance}, let $\a : [a,b] \to W$ be a future timelike curve, such that $\nabla \tau$ is defined at $\a(s)$ for almost all parameter values $s \in [a,b]$. Since $\tau \circ \a$ is increasing on $[a,b]$, basic analysis gives
$$\tau(y) - \tau(x) \ge \int_a^b (\tau \circ \a)'(s) ds$$
Furthermore, for almost all $s \in [a, b]$, $\nabla \tau$ is defined at $\a(s)$ and we have
$$(\tau \circ \a)'(s) = g(\nabla \tau, \a')$$
Putting these together with the estimate above, we have
$$\tau(y) - \tau(x) \ge \int_a^b g(\nabla \tau, \a')ds \ge \int_a^b C||\a'||_hds = C L_h(\a)\ge C d_h(x,y)$$
Finally, for $x, y \in W$, with $x \le y$, let $y \ll y_j  \in W$, $y_j \to y$. Then we have $x \ll y_j$, and hence $\tau(y_j) - \tau(x) \ge C d_h(x,y_j)$, as above. Since $\tau$ and $d_h$ are continuous, taking the limit as $j \to \8$ gives $\tau(y) - \tau(x) \ge C d_h(x,y)$.
\end{proof}

\vspace{1pc}
We note finally that the proof of Theorem \ref{bddgradthm} shows that the `(a) $\Longrightarrow$ (b)' part of Proposition \ref{revLipdiffble} also holds for time functions $\tau$ which are (only) differentiable almost everywhere. Since the original proof of the converse remains valid in this case, we note that the regularity assumption in Proposition \ref{revLipdiffble} may be relaxed accordingly.

\section{Cosmological Time} \label{seccosmo}

We now focus on what may be viewed as a canonical time function for spacetimes emanating from an `initial singularity', defined by Andersson, Galloway, and Howard in \cite{AGHcosmo} as follows: 

\begin{Def} [\cite{AGHcosmo}] Let $(M,g)$ be a spacetime, with Lorentzian distance $d_g$. The \emph{cosmological time function} $\tau_g$ of $M$ is defined by
$$\tau_g(q) := \sup_{p \le q} \, d_g(p,q)$$ 
Letting $\mathcal{C}^-(q)$ denote the set of all past causal curves from $q$, then equivalently:
$$\tau_g(q) = \sup \, \{L_g(\a) : \a \in \mathcal{C}^-(q) \}$$
\end{Def}

\vspace{1pc}
(Wald and Yip previously used the time dual of this in \cite{WaldYip}, there called the `\emph{maximum lifetime function}'.) Clearly, $\tau_g \equiv \infty$ on Minkowski. On the other hand, on any spacetime with finite past, $\tau_g$ is finite, and increases (weakly) on future causal curves. Simple examples show however that, for one, $\tau_g$ need not be continuous in this case. In \cite{AGHcosmo}, $\tau_g$ is said to be \emph{regular} if $\tau_g$ is finite-valued on all of $M$, and $\tau_g \to 0$ along every past-inextendible causal curve. We recall some facts from \cite{AGHcosmo}:

\begin{thm}  [\cite{AGHcosmo}] \label{AGHcosmothm} Suppose $M$ has regular cosmological time $\tau_g$. Then we have:
\ben
\item [(1)] $\tau_g$ is a (continuous) time function on $M$, satisfying
$$x \le y \; \; \Longrightarrow \; \; \tau_g(y) - \tau_g(x) \ge d_g(x,y)$$
\item [(2)] The gradient $\nabla \tau_g$ exists almost everywhere.
\item [(3)] For each point $q \in M$, there is a future-directed unit-speed timelike geodesic $\g_q : (0, \tau_g(q)] \to M$, called a `generator', along which $\tau_g(\g_q(t)) = t$, ending at $\g_q(\tau_g(q)) = q$.
\item [(4)] The set of all terminal tangent vectors $\mathcal{T} = \{\g_q'(\tau_g(q)) : q \in M\}$ is locally bounded away from the light cones.
\een
\end{thm}

\vspace{.5pc}
Note that, in general, there may be several generators to a given point $q \in M$. (For a simple example, take two points $p_1$ and $p_2$ in a common time-slice in Minkowski, and consider the spacetime $M = I^+(p_1) \cup I^+(p_2)$.) The following result shows that a regular $\tau_g$ can not be differentiable at such points.

\begin{prop} \label{cosmograd} Suppose $\tau_g$ is regular. If $\nabla \tau_g$ exists at $q \in M$, then there is a unique generator $\g_q$ to $q$, and $(\nabla \tau_g)_q = - \g_q'(\tau_g(q))$. In particular, wherever it is defined, $\nabla \tau_g$ is past timelike unit.
\end{prop}

\begin{proof} Let $\tau = \tau_g$, and fix any $q \in M$ at which $\nabla \tau$ exists. Let $\g = \g_q$ be any generator to $q$. Then $u := \g'(\tau(q))$ is a future timelike unit vector at $q$, and at this point we may write $\nabla \tau = -au + bv$, where $v$ is a spacelike unit vector orthogonal to $u$, and $a, b \ge 0$. Since $\g$ is a generator for $\tau$, we have
$$a = g(\nabla \tau, \g') = \frac{d}{dt}(\tau (\g(t))) = \frac{d}{dt}(t) = 1$$
Hence, $\nabla \tau = - u + bv$, and since this must be causal, we have $0 \le b \le 1$. In particular, this gives $||\nabla \tau||_g = \sqrt{1 - b^2} \le 1$. On the other hand, fix any future-pointing timelike unit vector $V \in T_qM$, and let $\a : (-\e, \e) \to M$ be a future timelike curve with $\a'(0) = V$, parameterized with respect to (Lorentzian) arc length. Then using the `reverse Lipschitz' condition in (1) of Theorem \ref{AGHcosmothm}, we have
\be
g(\nabla \tau, V)= \lim_{\; t \to 0^+} \frac{\tau(\a(t)) - \tau(\a(0))}{t} \ge \lim_{\; t \to 0^+} \frac{d_g(\a(0), \a(t))}{t} \ge 1\nonumber
\ee
We will show that this fact implies $b = 0$, from which the result follows. First note that if $b \not = 1$, and hence $\nabla \tau$ is not null, then we may apply the above to $V = - \nabla \tau / ||\nabla \tau||_g$, which gives
$||\nabla \tau||_g = g(\nabla \tau, V)  \ge 1$. Hence, $||\nabla \tau||_g = 1$, and $b = 0$. Suppose on the other hand that $b = 1$. Then for any $A \ge 1$, the vector $V = Au - (\sqrt{A^2-1}\,)v$ is a future timelike unit vector, and hence we have
$$1 \le g(\nabla \tau, V) = A - \sqrt{A^2-1}$$
But since the right-hand side goes to 0 as $A \to \infty$, this gives a contradiction.
\end{proof}

\vspace{.5pc}
Our main result on cosmological time now follows by combining Theorem \ref{AGHcosmothm} and Proposition \ref{cosmograd} with Theorem \ref{bddgradthm}:

\begin{thm} \label{cosmorevLip} Let $(M,g)$ be a spacetime with regular cosmological time function $\tau_g$. Then $\tau_g$ is locally anti-Lipschitz, and thus induces a definite null distance function on $M$.
\end{thm}

\section{Appendix}

Examples distinguishing the conditions in Lemma \ref{locsteepT} are given. We first note the following:

\begin{lem} \label{compRmetrics} Let $M$ be a manifold and $U \subset M$ open. Let $h_1, h_2$ be any two Riemannian metrics on $U$. Then for any precompact neighborhood $U_1 \subset \subset U$, there is a positive constant $C_1 > 0$ such that, for all vectors $X \in TU_1$, we have $||X||_{h_2} \ge C_1||X||_{h_1}$.
\end{lem}

\begin{exm} \label{steepspills} Consider the following vector field $T$ on the Minkowski plane, $\field{M}^{1+1} = \{(t,x)\}$.
\[
T := \left\{
        \begin{array}{ll}
            \, \, - \d_t  & \; \; \; \; \; \; x \in (-\8,-1] \cup \{0\} \cup [1, \8) \\ [.5pc]
            - \dfrac{1}{|x|}\d_t + \dfrac{\sqrt{1-x^2}}{x}\d_x &  \; \; \; \; \; \; x \in (-1,0) \cup (0,1) \\[.2pc]
        \end{array}
    \right.
\]
Then $T$ is everywhere past-pointing timelike, with $||T||_g = 1$. Let $X = \d_t - \d_x$. Then for $x \in (0,1)$, 
\be
g(T,X) = g\bigg(- \frac{1}{x}\d_t + \frac{\sqrt{1-x^2}}{x}\d_x \, , \, \d_t -\d_x\bigg) = \frac{1 -\sqrt{1-x^2}}{x} \nonumber
\ee
Let $U$ be any neighborhood of $(0,0)$, and $h$ any Riemannian metric on $U$. Fix any precompact neighborhood $U_0$ of $(0,0)$, with $U_0 \subset \subset U$. Then, as in Lemma \ref{compRmetrics}, there is a $C >0$ such that, for all vectors $X \in TU_0$, we have $||X||_h \ge C||X||_{\field{E}}$, where the latter is the norm with respect to the standard Euclidean metric on $\field{R}^2$. We have $||X||_{\field{E}} = \sqrt{2}$, and hence
\be
\dfrac{g(T,X)}{||X||_h} \le \dfrac{g(T,X)}{C||X||_{\field{E}}} = \frac{1 -\sqrt{1-x^2}}{\sqrt{2}Cx} \nonumber
\ee
Since the right hand side goes to 0 as $x \to 0$, condition (b) in Lemma \ref{locsteepT} fails.
\end{exm}

\begin{exm} \label{tipsnospill}  Consider the following vector field $T$ on the Minkowski plane:
\[
T := \left\{
        \begin{array}{ll}
            - \sqrt{2} \bigg((j+1)\, \d_t + j \d_x \bigg)& \; \; \; (t,x) = (0,1/j), \, j \in \field{N}\\ [.8pc]
             - \sqrt{2}\d_t &  \; \; \; (t,x) \ne (0,1/j) \\[.2pc]
        \end{array}
    \right.
\]
Hence, $T$ (or rather $-T$) `tips over' at $(0,0)$. In particular, condition (a) in Lemma \ref{locsteepT} fails. However, fix any nontrivial future causal vector $X = X^u\d_t + X^1\d_x$. Hence, $X^u >0$ and $X^u \ge |X^1|$. If $X \in T_pM$, with $p \ne (0,1/j)$, then
\be
\dfrac{g(T,X)}{||X||_\field{E}} = \dfrac{\sqrt{2} X^u}{\sqrt{(X^u)^2 + (X^1)^2}} \ge \dfrac{\sqrt{2}X^u}{\sqrt{2}X^u} = 1 \nonumber
\ee
If $X \in T_pM$, with $p = (0,1/j)$, then
\be
\dfrac{g(T,X)}{||X||_\field{E}} = \dfrac{\sqrt{2} [(j+1)X^u - jX^1]}{\sqrt{(X^u)^2 + (X^1)^2}} \ge \dfrac{\sqrt{2} [(j+1)X^u - jX^u]}{\sqrt{2}X^u} = 1 \nonumber
\ee
Let $U$ be any neighborhood of $(0,0)$, and $h$ any Riemannian metric on $U$. Let $U_0$ be any precompact neighborhood of $(0,0)$, with $U_0 \subset \subset U$. Then by Lemma \ref{compRmetrics} there is a constant $C>0$ such that, for all $X \in TU_0$, we have $||X||_\field{E} \ge C||X||_h$. Then, by the above, for any future causal $X \in TU_0$, we have
\be
\dfrac{g(T,X)}{||X||_h} \ge \dfrac{g(T,X)}{C^{-1}||X||_{\field{E}}} \ge C \nonumber
\ee
\end{exm}


\providecommand{\bysame}{\leavevmode\hbox to3em{\hrulefill}\thinspace}
\providecommand{\MR}{\relax\ifhmode\unskip\space\fi MR }
\providecommand{\MRhref}[2]{%
  \href{http://www.ams.org/mathscinet-getitem?mr=#1}{#2}}
\providecommand{\href}[2]{#2}

\end{document}